\documentclass[reqno]{amsart}

\usepackage{graphicx,pdfsync,amssymb, latexsym,amsfonts,amsbsy,color,esint,mathtools,nicefrac,cancel,bbm,tikz, caption, subcaption}
\usepackage[alphabetic,nobysame]{amsrefs}
\usepackage[shortlabels]{enumitem}

\usepackage{hyperref}
\hypersetup{
colorlinks=true,
linkcolor=blue,
filecolor=blue,      
urlcolor=blue,
citecolor=blue,
}
\newtheorem{theorem}{Theorem}[section]
\newtheorem{lemma}[theorem]{Lemma}
\newtheorem{proposition}[theorem]{Proposition}
\newtheorem{definition}[theorem]{Definition}

\newtheorem{remark}[theorem]{Remark}

\numberwithin{equation}{section}

\usepackage{times}

\let\div\relax
\DeclareMathOperator\div{div}

\DeclareMathOperator\Id{Id}
\DeclareMathOperator\supp{supp}

\newcommand\newD{\tilde{\mathcal{D}}}

\title[Blowup for a flow with passive tracer]{Instantaneous blowup of incompressible flow with passive tracer}

\author [Mimi Dai]{Mimi Dai}

\address{Department of Mathematics, Statistics and Computer Science, University of Illinois at Chicago, Chicago, IL 60607, USA}
\email{mdai@uic.edu} 

\author [Xiaotong ``Dawson" Yang]{Xiaotong ``Dawson" Yang}

\address{Department of Mathematics, Statistics and Computer Science, University of Illinois at Chicago, Chicago, IL 60607, USA}
\email{xyang212@uic.edu} 

\thanks{The authors are grateful for the support of the NSF grant DMS--2308208. M.Dai also acknowledges the support from Simons Foundation. }

\begin{document}

\begin{abstract}
We construct a family of solutions $(u,b)$ of the incompressible flow with a passive tracer for which both $\|u(t)\|_{L^\infty}$ and $\|b(t)\|_{L^\infty}$ blow up at time $T_*$. Away from $T_*$, the solutions remain smooth in both space and time. The argument adapts the inverse cascade mechanism from \cite{CDP} to the presence of an advected scalar, but the passive component creates a new compatibility constraint: the iteration must propagate the tracer while preserving the same principal velocity profiles from one stage to the next. We resolve it by introducing a simultaneous decomposition lemma for a symmetric tensor and a vector field.

We also show the existence of an infinitely family of instantaneous blowup solutions to the 2D MHD system, with critical blowup rate for the velocity component according to the scaling of the system. Moreover, the non-uniqueness is sharp in the sense that it occurs in spaces borderline to $L^2_tL_x^\infty$, the endpoint space of the Ladyzhenskaya--Prodi--Serrin type where uniqueness is known.


\end{abstract}

\maketitle

\begingroup
\setlength{\parskip}{0pt}
\tableofcontents
\endgroup

\section{Introduction}

We consider the system of incompressible flow with a passive tracer 
\begin{equation}\label{eq-main}
\begin{split}
\partial_t u -   \Delta u + \div( u \otimes u) + \nabla p &= 0, \\
\partial_t b -   \Delta b + (u\cdot\nabla)b  &= 0, \\
\div u  & =0
\end{split}
\end{equation}
on $\mathbb T^d\times [0,T]$ or $\mathbb R^d\times [0,T]$ with $d\geq 2$ and some $T>0$. In the coupled system \eqref{eq-main}, $u$ is the velocity field and $p$ is the pressure function; while $b$ is the scalar-valued tracer function. 


Interestingly, a particular context of the incompressible magnetohydrodynamic (MHD) system can be viewed as a system of a flow with passive tracer. 
The classical incompressible MHD system is governed by
\begin{equation*}
\begin{split}
\partial_t u -   \Delta u + \div( u \otimes u-B \otimes B) + \nabla p &= 0, \\
\partial_t B -   \Delta B + \div( u \otimes B-B \otimes u)  &= 0, \\
\div u = 0 & =0.
\end{split}
\end{equation*}
Assume in particular
\[u=(u^1(x_1,x_2), u^2(x_1,x_2), 0), \qquad B=(0,0,b(x_1,x_2)),\]
which implies
\[\div (B \otimes B)=0, \qquad \div( u \otimes B-B \otimes u)=(u\cdot\nabla)b.\]
Therefore the MHD system above reduces to \eqref{eq-main} in the 2D case, which is sometimes referred as two and half dimension (2.5D) MHD (cf.\cite{BLFNL}). 

Moreover, the 2D MHD with $u=(u_1(x_1,x_2), u_2(x_1,x_2))$ and $B=(B_1(x_1,x_2), B_2(x_1,x_2))$ can be rewritten into
\begin{equation}\label{mhd-2d}
\begin{split}
\partial_t u -   \Delta u + \div( u \otimes u-B \otimes B) + \nabla p &= 0, \\
\partial_t b -  \Delta b + u\cdot\nabla b &= 0, \\
\div u  =0, \ B&=\nabla^\perp b,
\end{split}
\end{equation}
with $b=b(x_1,x_2)$ being a stream function for the magnetic field $B$. We note that \eqref{mhd-2d} is in the same form of \eqref{eq-main} with an additional forcing term $\div(B \otimes B)$.


It is known that for the first equation in \eqref{eq-main} - the Navier-Stokes equation, the existence of global classical solution remains an open problem in dimensions $d\geq 3$. Thus as a coupled system, it is also not known whether a classical solution to \eqref{eq-main} or \eqref{mhd-2d} would blow up at a finite time or not. It is therefore necessary to introduce the standard notion of weak solutions for \eqref{eq-main}. A similar definition of weak solutions for \eqref{mhd-2d} can be given as below.

\begin{definition}\label{def:weak_solutions}
Let $\mathcal{D}_T$ denote the class of pairs
\[
(\varphi,\tilde\varphi)\in C^\infty(\mathbb T^d\times\mathbb R;\mathbb R^d)\times C^\infty(\mathbb T^d\times\mathbb R)
\]
such that $\div\varphi=0$ and $\varphi(t)=\tilde\varphi(t)=0$ for $t\ge T$.
Given weakly divergence-free $u_0\in L^2(\mathbb T^d)$ and $b_0\in L^2(\mathbb T^d)$, we say that $(u,b)$ is a weak solution of \eqref{eq-main} with initial data $(u_0,b_0)$ provided that $u(t)$ is weakly divergence-free for a.e. $t\in [0,T]$ and 
\begin{equation*}
\begin{split}
\int_{\mathbb{T}^d} u_0(x)\cdot \varphi(x,0) \, dx &= - \int_0^T \int_{\mathbb{T}^d} u\cdot \big(  \partial_t \varphi+ \Delta \varphi +  u \cdot \nabla \varphi  \big) \, dx dt,\\ 
\int_{\mathbb{T}^d} b_0(x)\tilde \varphi(x,0) \, dx &= - \int_0^T \int_{\mathbb{T}^d} b\big(  \partial_t \tilde\varphi+ \Delta \tilde\varphi +  u \cdot \nabla \tilde\varphi  \big)\, dx dt
\end{split}
\end{equation*}
for every $(\varphi,\tilde\varphi)\in \mathcal{D}_T$.

\end{definition}

Equivalently, a weak solution of \eqref{eq-main} satisfies the Duhamel representation
\[
\begin{split}
u(t) &= e^{t\Delta}u_0 - \int_0^t e^{(t-s)\Delta} \mathbb{P}\div(u\otimes u) (s)\, ds,\\
b(t) &= e^{t\Delta}b_0 - \int_0^t e^{(t-s)\Delta} \div(u b) (s)\, ds
\end{split}
\]
in the sense of distributions; analogously, a weak solution of \eqref{mhd-2d} also satisfies
\[
\begin{split}
u(t) &= e^{t\Delta}u_0 - \int_0^t e^{(t-s)\Delta} \mathbb{P}\div(u\otimes u-\nabla^\perp b\otimes \nabla^\perp b) (s)\, ds,\\
b(t) &= e^{t\Delta}b_0 - \int_0^t e^{(t-s)\Delta} \div(u b) (s)\, ds.
\end{split}
\]

The scaling of \eqref{eq-main} is inherited from the Navier--Stokes equation in the velocity component, while the passive scalar allows an additional homogeneity parameter. More precisely, if $(u,p,b)(x,t)$ solves \eqref{eq-main}, then for every $\lambda>0$ the rescaled fields
\[u_\lambda(x,t)=\lambda u(\lambda x,\lambda^2t), \quad p_\lambda(x,t)=\lambda^2 p(\lambda x,\lambda^2t),\]
\[b_\lambda(x,t)=\lambda^{\mathfrak{a}} b(\lambda x,\lambda^2t), \quad \forall \mathfrak{a}\in\mathbb R,\]
again solve the same system with the corresponding rescaled initial data. The velocity scaling gives a family of critical spaces
\begin{equation}
\label{eq:critical_spaces}\dot H^{\frac d2-1}\subset L^d\subset L^{d,\infty}\subset \dot W^{-1, \infty}\subset BMO^{-1} \subset \dot B^{-1}_{\infty,\infty}.
\end{equation}
The same scaling leaves invariant the Ladyzhenskaya--Prodi--Serrin classes $L^p_tL^q_x$ satisfying $\frac2p + \frac{d}{q}=1$ with $q>d$. The freedom in choosing $\mathfrak{a}$ reflects that $b$ evolves by a linear advection-diffusion equation once $u$ is fixed. The choice $\mathfrak{a}=1$ recovers the scaling of the MHD system.

It is obvious that the 2D MHD system \eqref{mhd-2d} has the scaling
\[u_\lambda(x,t)=\lambda u(\lambda x,\lambda^2t), \quad B_\lambda(x,t)=\lambda B(\lambda x,\lambda^2t), \quad p_\lambda(x,t)=\lambda^2 p(\lambda x,\lambda^2t),\]
\[b_\lambda(x,t)=b(\lambda x,\lambda^2t).\]


We now formulate the main results for \eqref{eq-main}.

\begin{theorem} \label{main-thm}
Let $u_0\in C^\infty(\mathbb T^2)$ be divergence-free and let $b_0\in C^\infty(\mathbb T^2)$. Then there exists $T>0$ with the following property: for every $T_*\in[0,T)$, one can find a weak solution $(u(t), b(t))$ of \eqref{eq-main} on $\mathbb T^2\times[0,T]$ such that:
\begin{enumerate}[1.,ref=\arabic*,left=1em]
    \item On the interval $[0,T_*]$, the pair $(u(t), b(t))$ is a classical solution of \eqref{eq-main} with initial data $(u(0), b(0))=(u_0,b_0)$;\label{blow_up_theorem_classical_1}
    \item On $(T_*,T]$, the pair $(u(t), b(t))$ is again a classical solution of \eqref{eq-main} and there exists $c>0$ such that
    \[
    \|u(t_n)\|_{L^\infty} \geq \frac{c}{\sqrt{t_n-T_*}}, \qquad  \|\nabla u(t_n)\|_{L^\infty} \geq \frac{c}{t_n-T_*},
    \]
     \[
    \|b(t_n)\|_{L^\infty} \geq \frac{c}{\sqrt{t_n-T_*}}, \qquad  \|\nabla b(t_n)\|_{L^\infty} \geq \frac{c}{t_n-T_*},
    \]
    along a sequence of times $t_n \to T_*+$;\label{blow_up_theorem_classical_2_with_lower_bound}
    moreover, there is another constant $C>0$ such that
    \[
    \|u(t)\|_{L^\infty}+\|b(t)\|_{L^\infty}  \leq \frac{C}{\sqrt{t-T_*}}, \quad  \|\nabla u(t)\|_{L^\infty} + \|\nabla b(t)\|_{L^\infty}\leq \frac{C}{t-T_*}
    \]
    for all $t\in(T_*,T]$; \label{blow_up_theorem_type_i}
    \item The trajectory $(u(t), b(t))$ is weak-* continuous in time with values in $BMO^{-1}\times BMO^{-1}$;\label{blow_up_theorem_continuity}
    \item The solution also satisfies
    \begin{equation}
    u, b\in L^2([0,T];L^p), \quad \forall p<\infty.
\end{equation}
\label{blow_up_theorem_prodi_serrin}
\end{enumerate}
\end{theorem}



\begin{theorem}\label{thm-nonunique-2d}
There exists an infinite family of such instantaneous blowup solutions given in Theorem \ref{main-thm} from the same initial data. 
\end{theorem}

Moreover, the results hold in 3D and higher dimensions. Specifically, we have
\begin{theorem} \label{thm-3d}
For any $u_0\in C^\infty(\mathbb T^3)$ with $\div u_0=0$ and $b_0\in C^\infty(\mathbb T^3)$, there exist $T>0$ and infinitely many weak solutions $(u(t), b(t))$ of \eqref{eq-main} on $\mathbb T^3\times[0,T]$, such that for any $T_*\in[0,T)$, the weak solutions $(u(t), b(t))$ satisfy all the properties of Theorem \ref{main-thm}. In addition, $u,b\in L_t^\infty \dot W^{-1,\infty}_x$. 
\end{theorem}

For the 2D MHD system \eqref{mhd-2d}, we have 

\begin{theorem} \label{thm-2dmhd}
Let $u_0, B_0\in C^\infty(\mathbb T^2)$ with $\div u_0=\div B_0=0$. There exists $T>0$ with the following property: for every $T_*\in[0,T)$, there exist infinitely many weak solutions $(u(t), B(t))$ of \eqref{mhd-2d} on $\mathbb T^2\times[0,T]$ with non-trivial $B(t)$ such that:
\begin{enumerate}[1.,ref=\arabic*,left=1em]
    \item The pair $(u(t), B(t))$ is a classical solution of \eqref{mhd-2d} on $[0,T_*]$ and $(T_*,T]$ separately with initial data $(u(0), B(0))=(u_0,B_0)$; and $(u(t), B(t))$ is weak-* continuous in time with values in $BMO^{-1}\times BMO^{-1}$;
    \item On $(T_*,T]$, the pair $(u(t), B(t))$ satisfies
    \[
    \|u(t_n)\|_{L^\infty} \gtrsim \frac{1}{\sqrt{t_n-T_*}}, \qquad  \|\nabla u(t_n)\|_{L^\infty} \gtrsim \frac{1}{t_n-T_*},
    \]
    along a sequence of times $t_n \to T_*+$, and
    \[
    \|u(t)\|_{L^\infty}+\|B(t)\|_{L^\infty}  \lesssim \frac{1}{\sqrt{t-T_*}}, \quad  \|\nabla u(t)\|_{L^\infty} + \|\nabla B(t)\|_{L^\infty}\lesssim \frac{1}{t-T_*}
    \]
    for all $t\in(T_*,T]$; 
    \item Moreover,
    \begin{equation}\notag
    u, B\in L^2([0,T];L^p), \quad \forall p<\infty.
\end{equation}
\end{enumerate}
\end{theorem}

As discussed in the companion paper \cite{Dai2026}, the instantaneous blowup rate established above is critical according to the scaling of the systems. The non-uniqueness occurs in spaces borderline to $L^2([0,T];L^\infty)$ which is the endpoint space of the Ladyzhenskaya--Prodi--Serrin type, where uniqueness is known. Therefore, the non-uniqueness is sharp in this sense.

As seen in the work \cite{CDP}, the proof of the 2D case is more technical than that of the 3D (or higher dimension) case. Thus, we will present the proof of Theorem \ref{main-thm} in the following sections.
The proof of Theorem \ref{thm-3d} is similar to that of Theorem 1.1 in \cite{Dai2026} and thus is omitted in this paper. Also, we focus on the blowup phenomenon; the non-uniqueness stated in Theorem \ref{thm-nonunique-2d} follows from analogous analysis as in \cite{CDP}.

The construction leading to a proof of Theorem \ref{thm-2dmhd} for the 2D MHD system \eqref{mhd-2d} is provided in Section \ref{sec:perturbative-2dmhd}.


\section{Notations and technical lemmas}\label{sec:notation}

\subsection{Notations}
Throughout the paper,
\[
\mathbb T^2=(\mathbb R/2\pi\mathbb Z)^2.
\]
For $1\leq p\leq\infty$, the space $L^p(\mathbb T^2)$ is understood in the usual sense, and we write either $\|f\|_{L^p(\mathbb T^2)}$ or simply $\|f\|_p$. We use $L^{p,\infty}$ for the Lorentz, or weak $L^p$, space, and we write $|\Omega|$ for the measure of $\Omega$ inside a fixed fundamental domain of $\mathbb T^2$. The spaces $BMO^{-1}$ and $X_T$ are taken in the sense of \cites{KochTataru2001,CDP}. We also write $\|\cdot\|_{C^\alpha}$ for the $\alpha$-H\"older seminorm on either $\mathbb R^2$ or $\mathbb T^2$.

We next record the Fourier multipliers that will be used repeatedly. The heat semigroup and the Leray projection are given by
$$e^{t\Delta}f\coloneqq\mathcal F^{-1}(e^{-t|\xi|^2}\hat f(\xi)),$$
$$\mathbb Pf\coloneqq \mathcal F^{-1}\left(\left(\Id-\frac{\xi\otimes \xi}{|\xi|^2}\right)\hat f(\xi)\right),$$
for vector-valued $f$. For dyadic $N\in 2^\mathbb N$, we also define the Littlewood--Paley projections
$$P_Nf\coloneqq\mathcal F^{-1}(\psi(|\xi|/N)\hat f(\xi)),$$
where $\psi\in C_c^\infty((2/3,3/2))$ is chosen so that $\sum_N\psi(r/N)\equiv 1$ for all $r\in[3/4,\infty)$.

For a vector field $f$, let $\nabla\odot f$ denote the symmetric tensor with entries $\frac12(\partial_if_j+\partial_jf_i)$. Following \cites{CDP,CoiculescuPalasek2025}, we introduce
\begin{align*}
    \mathcal Df&\coloneqq2\nabla\odot f-2(\div f)\Id,\\
    \newD f&\coloneqq2\nabla\odot f-(\div f)\Id,\\
    \mathcal R&\coloneqq \Delta^{-1}\newD,\\
    \mathbb Q&\coloneqq2\Delta^{-1}\nabla\odot \mathbb P\div,
\end{align*}
and recall the identities
\begin{align*}
    \div \mathcal D&=\Delta\mathbb P,\\
    \div\newD &= \Delta,\\
    \newD&=\mathbb Q\mathcal D+\left(2\frac{\nabla\otimes\nabla}{\Delta}-\Id\right)\div,\\
    \div \mathcal R&=P_{\neq0},
\end{align*}
where $P_{\neq0}$ denotes the projection onto nonzero Fourier modes, together with
\begin{align}\label{eq:Q_identities}
    \mathbb Q=\mathcal R\mathbb P\div,\qquad \mathbb Q\mathcal D=2\nabla\odot\mathbb P.
\end{align}

\subsection{Technical lemmas}
The following technical lemmas were proved in \cite{CDP}, and will be used frequently in later sections. 
\begin{lemma}[Commutator estimate for $e^{t\Delta}$]\label{l:commutator}
    Let $a\in C^\infty(\mathbb T^d)$ and $\xi\in\mathbb Z^d\setminus \{0\}$, and define
    \begin{align*}
        A_i\coloneqq |\xi|^{-i}\|\nabla^ia\|_{L^\infty}.
    \end{align*}
    Then
    \begin{equation}\begin{aligned}\label{commutator_inequality}
        &\|\nabla^n[e^{t\Delta},a(x)]\sin(\xi\cdot x)\|_{L^\infty}\\
        &\qquad\lesssim_{m,n} |\xi|^{n}\sum_{i=0}^n \left((A_i^{1-1/m}A_{m+i}^{1/m}+A_i^{1-2/m}A_{m+i}^{2/m})e^{-|\xi|^2t/4}+A_{m+i}\right)
        \end{aligned}\end{equation}
        for $m\geq3+n$. Furthermore,
        \begin{align}\label{heat-decay-estimate}
             \|\nabla^n e^{t\Delta}(a(x)\sin(\xi\cdot x))\|_\infty&\lesssim_{m,n} |\xi|^{n}\left(\sum_{i=0}^{n}A_i e^{-|\xi|^2t/4}+\sum_{i=m}^{n+m}A_{i}\right).
        \end{align}
\end{lemma}

\begin{lemma}\label{l:heat_stationary_phase}
Let $\xi\in\mathbb Z^d\setminus \{0\}$, $a\in C^\infty(\mathbb T^d; \mathbb R^d)$, and denote $A_{i,\kappa}\coloneqq|\xi|^{-i-\kappa}\|\nabla^ia\|_{C^\kappa}$ for $\kappa\in(0,1)$. Then for $n\geq 0$ and $m\geq 1$
\begin{align*}
\|\nabla^ne^{t\Delta}\mathcal R\mathbb P(ae^{i\xi\cdot x})\|_{C^\kappa}&\lesssim |\xi|^{n+\kappa}\left(|\xi|^{-1}\sum_{i=0}^{m-1}A_{i,\kappa}e^{-|\xi|^2t/4}+\sum_{i=m}^{n+2m}
A_{i,\kappa}\right).
\end{align*}
\end{lemma}

\begin{lemma}[Oscillation estimate]\label{l:oscillation_estimate}
Let $a\in C^\infty(\mathbb T^d;\mathbb R^d)$ and $b\in C^\infty(\mathbb T^d;\mathbb R)$. Let $\lambda\in\mathbb N$ and $\Xi\in \mathbb Z^d$ with $|\Xi|\in\frac12\lambda\mathbb N$. Then, for any $n\geq0$, $m\geq1$, and $\kappa\in(0,1)$, we obtain 
\begin{align*}
&\Big\|\nabla^ne^{t\Delta}\mathcal R\mathbb P\Big(a(x)\Big(b(\lambda x)\sin^2(\Xi\cdot x)-\fint_{\mathbb T^d}b(\lambda y)\sin^2(\Xi\cdot y)dy\Big)\Big)\Big\|_{C^\kappa}\\
&\qquad\lesssim_{m,n,\kappa} \Bigg(\sum_{i=0}^{m-1}(\lambda^{n-i-1}+|\Xi|^{n-i-1})\|a\|_{C^{i,\kappa}}e^{-\lambda^2t/4}\\
&\qquad\qquad\qquad\qquad+\sum_{i=m}^{n+2m}
(\lambda^{n-i}+|\Xi|^{n-i})\|a\|_{C^{i,\kappa}}\Bigg)\|\nabla^{n+2d}b\|_{L^\infty}
\end{align*}
\end{lemma}

\section{Heuristics of the scheme}
\label{sec:heuristics}

\subsection{Main ideas of the construction}\label{sec:idea}
At the heuristic level, the blowup is driven by energy injected from arbitrarily high frequencies into low frequencies. Accordingly, the construction is organized as an inverse cascade from high modes to low modes. This cascade is produced iteratively by running convex integration along a sequence of times converging to the blowup time, which in the proofs is shifted from $T_*$ to $t=0$. The underlying cascade mechanism is the same as in the Navier--Stokes construction of \cite{CDP}, but the coupled structure of \eqref{eq-main} changes what must be preserved from one stage to the next. Each step must both generate the desired low-frequency contribution and keep the principal ansatz intact. The high-high to low transfer is a standard feature of convex integration; the genuinely new issue is enforcing this compatibility for a system involving both a symmetric tensor and a vector field. This is why we need a new decomposition lemma adapted to the nonlinear structure of \eqref{eq-main}.


We first recall the standard geometric lemma (see, for instance, \cite{DeLellisSzekelyhidi2013}) for a symmetric tensor which has been extensively used in previous convex integration constructions. 

\begin{lemma}[Symmetric geometric lemma]\label{le-geometry}
Let $B_{\varepsilon_u}(Id)$ be the ball of radius $\varepsilon_u>0$ centered at the identity matrix in the space of $2\times 2$ symmetric matrices. The exists a finite subset $\Lambda_u\subset S^1\cap \mathbb Q^2$ of vectors $\eta$ associated with orthonormal bases $(\eta, \eta^{\perp})$ such that for any $R\in B_{\varepsilon_u}(Id)$, we have the decomposition 
\begin{equation}\label{decomp1}
R=\sum_{\eta\in \Lambda_u}\Gamma_\eta^2(R) \eta^{\perp}\otimes \eta^{\perp}
\end{equation}
for smooth functions $\Gamma_\eta: B_{\varepsilon_u}(Id)\to \mathbb R$.
\end{lemma}

We fix the notation: if $\eta=(\eta_1,\eta_2)$, $\eta^{\perp}=(-\eta_2, \eta_1)$. 

Note $\Lambda_u$ in Lemma \ref{le-geometry} can be chosen such that $\Lambda_u\cap \{\pm e_1,\pm e_2\}=\emptyset$. In addition to the standard geometric lemma above, we prove the following decomposition lemma, which can be generalized to any higher dimensional case.

\begin{lemma}[ Tensor-vector decomposition]\label{le-geometry-TV}
There exist $\varepsilon_*>0$ and a finite family of unit vectors
\[
\Lambda_b=\{\eta^m\}_{m=1}^N\subset S^1\cap\mathbb Q^2
\]
together with smooth maps
\[
\Gamma_m:\;B_{\varepsilon_*}(\Id)\times \mathbb R^2\to (0,\infty),\qquad
p:\;B_{\varepsilon_*}(\Id)\times \mathbb R^2\to \mathbb R,
\]
such that for every $(R,g)\in B_{\varepsilon_*}(\Id)\times\mathbb R^2$,
\[
\begin{split}
R-p(R,g)\Id&=\sum_{m=1}^N\Gamma_m^2(R,g)\,\eta^{m,\perp}\otimes\eta^{m,\perp},\\
g&=\sum_{m=1}^N\Gamma_m(R,g)\,\eta^{m,\perp}.
\end{split}
\]
\end{lemma}

\begin{proof}
By Lemma~\ref{le-geometry}, there exist $\varepsilon_u>0$, a finite set
$\Lambda_u\subset S^1\cap\mathbb Q^2$, and smooth functions
$A_\xi:B_{\varepsilon_u}(\Id)\to(0,\infty)$ such that
\[
R=\sum_{\xi\in\Lambda_u}A_\xi^2(R)\,\xi^\perp \otimes\xi^\perp,\qquad R\in B_{\varepsilon_u}(\Id).
\]
Fix $\varepsilon_*=\varepsilon_u$.
Denote $\Lambda_u^-=\{-\xi:\xi\in\Lambda_u\}$.
Define the family $\Lambda_b=\Lambda_u\cup \Lambda_u^- \cup \{\pm e_1,\pm e_2\}$.
We index this family by $m=1,\dots,N$.

For each $\xi\in\Lambda_u$, set
\[
\Gamma_{\xi,+}(R,g)\coloneqq \frac1{\sqrt2}A_\xi(R),\qquad
\Gamma_{\xi,-}(R,g)\coloneqq \frac1{\sqrt2}A_\xi(R).
\]
Then
\[
\Gamma_{\xi,+}^2\,\xi^\perp\otimes\xi^\perp+\Gamma_{\xi,-}^2\,(-\xi)^\perp\otimes(-\xi)^\perp
=A_\xi^2(R)\,\xi^\perp\otimes\xi^\perp,
\]
while the associated vector contribution vanishes:
\[
\Gamma_{\xi,+}\xi^\perp+\Gamma_{\xi,-}(-\xi)^\perp=0.
\]

Now write $g=(g_1,g_2)$ and define
\[
M(g)\coloneqq 1+|g|^2,\qquad
s_i(g)\coloneqq \sqrt{2M(g)-g_i^2},\quad i=1,2.
\]
Since $2M(g)-g_i^2=2+2|g|^2-g_i^2\ge 2+|g|^2>0$, each $s_i$ is smooth.
Set
\[
\begin{split}
\alpha_1(g)\coloneqq \frac{-g_1-s_1(g)}2,\qquad
\beta_1(g)\coloneqq \frac{g_1-s_1(g)}2,\\
\alpha_2(g)\coloneqq \frac{g_2+s_2(g)}2,\qquad
\beta_2(g)\coloneqq \frac{s_2(g)-g_2}2,
\end{split}
\]
and define
\[
\begin{split}
\Gamma_{e_1}(R,g)\coloneqq \alpha_2(g),\qquad
\Gamma_{-e_1}(R,g)\coloneqq \beta_2(g),\\
\Gamma_{e_2}(R,g)\coloneqq \alpha_1(g),\qquad
\Gamma_{-e_2}(R,g)\coloneqq \beta_1(g).
\end{split}
\]
Then
\[
\alpha_1-\beta_1=-g_1,\quad \alpha_2-\beta_2=g_2;  \quad
\alpha_i^2+\beta_i^2=M(g), \quad \mbox{for} \quad i=1,2.
\]
Thus we can verify
\[
\sum_{i=1}^2\big(\Gamma_{e_i}e_i^\perp+\Gamma_{-e_i}(-e_i)^\perp\big)=g
\]
and
\[
\sum_{i=1}^2\big(\Gamma_{e_i}^2 e_i^\perp\otimes e_i^\perp+\Gamma_{-e_i}^2 (-e_i)^\perp\otimes(-e_i)^\perp\big)=M(g)\Id.
\]

Combining both groups of directions gives
\[
\sum_{m=1}^N\Gamma_m(R,g)\,\eta^{m,\perp}=g,
\]
and
\[
\sum_{m=1}^N\Gamma_m^2(R,g)\,\eta^{m,\perp}\otimes\eta^{m,\perp}
=R+M(g)\Id.
\]
Therefore, setting
\[
p(R,g)\coloneqq -M(g),
\]
we obtain
\[
R-p(R,g)\Id=\sum_{m=1}^N\Gamma_m^2(R,g)\,\eta^{m,\perp}\otimes\eta^{m,\perp},\qquad
g=\sum_{m=1}^N\Gamma_m(R,g)\,\eta^{m,\perp}.
\]
It is clear to see that all coefficient maps are smooth on
$B_{\varepsilon_*}(\Id)\times\mathbb R^2$.
\end{proof}

We index the elements of $\Lambda_u$ together with $\Lambda_b$ despite repetitions by $1, 2, ..., J$ (so that $J$ is the sum of the numbers of elements in $\Lambda_u$ and $\Lambda_b$) and denote $\mathcal J=\{1, 2, ... , J\}$.
For $k\in\mathbb N$ and $j\in \mathcal J$, the frequency scales $N_{j,k}$ will be chosen to rapidly increase and satisfy
\[
A^{\mathfrak{b}^k}\sim N_{1,k} \ll N_{2,k} \ll \dots \ll N_{J,k} \ll N_{1,k+1} \sim A^{\mathfrak{b}^{k+1}}
\]
for a large constant $A\gg 1$ and another constant $\mathfrak{b}>1$.
Accordingly we choose a decreasing sequence of times $N_{1, k+1}^{-2} \ll t_k \ll N_{J,k}^{-2}$. We then design an iterative scheme such that the energy transfers from the $(k+1)$-th frequency level to the $k$-th level on the time interval $[t_{k+1},t_k]$.

\subsection{Building blocks and the ansatz of the principal solution for \eqref{eq-main}}
\label{sec-ansatz}
We begin with oscillatory profiles of the form
\[
\begin{split}
\psi_{j,k,u}(x)&\approx N_{j,k}^{-2}a_{j,k,u}(x)\eta^{j,\perp}\sin(N_{j,k}\eta^j\cdot x), \quad \eta^j\in \Lambda_u,\\
\psi_{j,k,c}(x)&\approx N_{j,k}^{-2}a_{j,k,b}(x)\eta^{j,\perp}\sin(N_{j,k}\eta^j\cdot x), \quad \eta^j\in \Lambda_b,\\
\psi_{j,k,b}(x)&\approx N_{j,k}^{-2}\sin(N_{j,k}\eta^j\cdot x), \quad \eta^j\in \Lambda_b
\end{split}
\]
with $N_{j,k}\eta^j\in\mathbb Z^2$ for all $j\in \mathcal J$ and $k\in\mathbb N$. The label ``c'' in $\psi_{j,k,c}$ marks the extra profile inserted to transmit the velocity-scalar coupling. The amplitudes $a_{j,k,u}$ and $a_{j,k,b}$ are chosen recursively from the previous frequency level and vary on spatial scales much larger than $N_{j,k}^{-1}$.

We decompose the solution as $(u,b)=(v,h)+(w,\zeta)$, where $(v,h)$ is the principal profile and $(w,\zeta)$ is the perturbative correction. The discussion below concerns only $(v,h)$ and explains how the inverse cascade is encoded in the main ansatz. We set heuristically
\[v(x,t) = \sum_k v_k(x,t), \quad h(x,t) = \sum_k h_k(x,t)\]
with components $v_k$ and $h_k$ given approximately by
\[
\begin{split}
v_k(x,t)& \approx -\sum_{j\in \Lambda_u} N_{j,k} e^{-N_{j,k}^2t} \mathbb P\Delta \psi_{j,k,u}(x)-\sum_{j\in \Lambda_b} N_{j,k} e^{-N_{j,k}^2t} \mathbb P\Delta \psi_{j,k,c}(x), \\
h_k(x,t)& \approx -\sum_{j\in \Lambda_b} N_{j,k} e^{-N_{j,k}^2t} \Delta \psi_{j,k,b}(x)
\end{split}
\]
for $t \geq t_k$,
where, as a shorthand, $j\in \Lambda_u$ and $j\in \Lambda_b$ mean that the sum runs over those indices for which $\eta^j$ belongs to the corresponding set. Because the amplitudes are much lower frequency than the oscillatory phases, the dominant effect of $-\Delta$ or $-\mathbb P\Delta$ is to remove the prefactor $N_{j,k}^{-2}$ and recover an $O(1)$ oscillation. More precisely,
$$-\mathbb P\Delta\psi_{j,k,u}(x)\approx a_{j,k,u}(x)\eta^{j,\perp}\sin(N_{j,k}\eta^j\cdot x)=O(1),$$ 
and likewise
\[-\Delta\psi_{j,k,b}(x)\approx O(1), \quad -\mathbb P\Delta \psi_{j,k,c}(x)\approx O(1).\]

The recursive goal is that the $(k+1)$-st frequency level generates the $k$-th one over the short interval $[t_{k+1},t_k]$. In heuristic form, we therefore want to choose $a_{j,k,u}$ and $a_{j,k,b}$ so that
\begin{equation} \label{ansatz-intro}
\begin{split}
v_k(x,t) &\approx -\int_0^{t} e^{(t-s)\Delta }\mathbb P\div ( v_{k+1}\otimes v_{k+1})(s) \, ds,\\
h_k(x,t) &\approx -\int_0^{t} e^{(t-s)\Delta }\div ( v_{k+1} h_{k+1})(s) \, ds.
\end{split}
\end{equation}
A direct computation of the time integral shows that the relevant scaling factor satisfies
\begin{equation}\label{scale-separation}
\int_0^{t} N_{j,k} N_{j',k} e^{-N_{j,k}^2 s-N_{j',k}^2 s} \, ds \sim 
\begin{cases}
1, &j=j',\\
N_{j,k}^{-1} N_{j',k} \ll 1, &j > j',
\end{cases}
\end{equation}
for $t\geq t_{k-1}$. This separation says that, to leading order, only the self-interactions from the same direction matter. Consequently,
\[
\begin{split}
&\quad\int_0^t P_{< N_{1,k+1}}\mathbb P\div( v_{k+1}\otimes v_{k+1}) \, ds \\
&\approx P_{< N_{1, k+1}}\mathbb P\div \sum_{j\in \Lambda_u} \mathbb P\Delta \psi_{j,k+1,u} \otimes \mathbb P\Delta \psi_{j,k+1,u}\\
&\quad+P_{< N_{1, k+1}}\mathbb P\div \sum_{j\in \Lambda_b} \left(\mathbb P\Delta \psi_{j,k+1,c} \otimes \mathbb P\Delta \psi_{j,k+1,c}\right)\\
& \approx \mathbb P\div\sum_{j\in\Lambda_u}a_{j,k+1,u}^2\eta^{j,\perp}\otimes\eta^{j,\perp}+\mathbb P\div\sum_{j\in\Lambda_b}a_{j,k+1,b}^2\eta^{j,\perp}\otimes\eta^{j,\perp}
\end{split}
\]
for $t\geq t_k$, after discarding the lower-order cross interactions. The scalar transport term admits the analogous leading-order reduction:
\[
\begin{split}
&\quad\int_0^t P_{< N_{1,k+1}}\div( v_{k+1} h_{k+1}) \, ds \\
&\approx P_{< N_{1, k+1}}\div \sum_{j\in \Lambda_b} \mathbb P\Delta \psi_{j,k+1,c} \Delta \psi_{j,k+1,b}\\
& \approx \div\sum_{j\in\Lambda_b}a_{j,k+1,b}\eta^{j,\perp}.
\end{split}
\]

Suppose now that the amplitudes are arranged so that
\begin{equation}\label{decomp-heuristics}
\begin{split}
\sum_{j\in \Lambda_u}a_{j,k+1,u}^2\eta^{j,\perp}\otimes\eta^{j,\perp} &\approx \mathcal D \sum_{j\in\Lambda_u} N_{j,k}\psi_{j,k,u}+p\Id, \\
\sum_{j\in\Lambda_b}a_{j,k+1,b}^2\eta^{j,\perp}\otimes\eta^{j,\perp}&\approx \mathcal D \sum_{j\in\Lambda_b}  N_{j,k}\psi_{j,k,c}+p\Id,\\
\sum_{j\in\Lambda_b}a_{j,k+1,b}\eta^{j,\perp} &\approx \nabla \sum_{j\in\Lambda_b}  N_{j,k}\psi_{j,k,b}
\end{split}
\end{equation}
for some scalar functions $p(x)$, which may change from line to line. Inserting these identities into the previous approximations yields
\[
\begin{split}
&\quad\int_0^t P_{< N_{1,k+1}}\mathbb P\div (v_{k+1}\otimes v_{k+1}) \, ds\\
& \approx \sum_{j\in\Lambda_u} N_{j,k} \Delta \psi_{j,k,u}+\sum_{j\in\Lambda_b} N_{j,k} \Delta \psi_{j,k,c} + \nabla p,
\end{split}
\]
\[\int_0^t P_{< N_{1,k+1}}\div( v_{k+1} h_{k+1}) \, ds\approx \sum_{j\in\Lambda_b} N_{j,k} \Delta \psi_{j,k,b} \]
for $t\geq t_k$. This is the heuristic mechanism behind \eqref{ansatz-intro}. 

The first line of \eqref{decomp-heuristics} falls within the standard geometric Lemma \ref{le-geometry}. The last two lines are more rigid: they must be realized simultaneously with the same coefficients $a_{j,k+1,b}$ because both come from the scalar-coupling component of the scheme. This is precisely where the new tensor-vector decomposition Lemma \ref{le-geometry-TV} is needed, and it is the genuinely new structural input in the construction.

\subsection{Heuristics for 2D MHD \eqref{mhd-2d}}
\label{sec-ansatz-2dmhd}
In this section we shall argue that the inverse-cascade mechanism combined with the convex integration approach described above seems to be rather difficult to be implemented for the genuine 2D MHD due to the rigid geometry in the magnetic equation.
For the same $N_{j,k}$ and $\eta^j$ as in previous section, now we choose the potential profiles
\[
\begin{split}
\psi_{j,k,u}(x)&\approx N_{j,k}^{-2}a_{j,k,u}(x)\eta^{j,\perp}\sin(N_{j,k}\eta^j\cdot x), \quad \eta^j\in \Lambda_u,\\
\psi_{j,k,c}(x)&\approx N_{j,k}^{-2}a_{j,k,b}(x)\eta^{j,\perp}\sin(N_{j,k}\eta^j\cdot x), \quad \eta^j\in \Lambda_b,\\
\psi_{j,k,b}(x)&\approx N_{j,k}^{-2}\tilde a_{j,k,b}(x)\cos(N_{j,k}\eta^j\cdot x), \quad \eta^j\in \Lambda_b
\end{split}
\]
with slowly varying amplitude functions $a_{j,k,u}$, $a_{j,k,b}$ and $\tilde a_{j,k,b}$ to be determined recursively.

As before the solution $(u,b)$ of \eqref{mhd-2d} would consist of the principal part $(v,h)$ and perturbative correction $(w,\zeta)$.
The pair $(v,h)$ is constructed as 
\[v(x,t) = \sum_k v_k(x,t), \quad h(x,t) = \sum_k h_k(x,t)\]
with $v_k$ and $h_k$ given by, up to leading order terms
\[
\begin{split}
v_k(x,t)& \approx -\sum_{j\in \Lambda_u} N_{j,k} e^{-N_{j,k}^2t} \mathbb P\Delta \psi_{j,k,u}(x)-\sum_{j\in \Lambda_b} N_{j,k} e^{-N_{j,k}^2t} \mathbb P\Delta \psi_{j,k,c}(x), \\
h_k(x,t)& \approx \sum_{j\in \Lambda_b} e^{-N_{j,k}^2t} \Delta \psi_{j,k,b}(x)
\end{split}
\]
for $t \geq t_k$. 

The amplitude functions $a_{j,k,u}$, $a_{j,k,b}$ and $\tilde a_{j,k,b}$ are chosen iteratively such that
\begin{equation} \label{ansatz-intro-2dmhd}
\begin{split}
v_k(x,t) &\approx -\int_0^{t} e^{(t-s)\Delta }\mathbb P\div ( v_{k+1}\otimes v_{k+1}-\nabla^\perp h_{k+1}\otimes \nabla^\perp h_{k+1})(s) \, ds,\\
h_k(x,t) &\approx -\int_0^{t} e^{(t-s)\Delta }\div ( v_{k+1} h_{k+1})(s) \, ds.
\end{split}
\end{equation}
Thanks to the scaling separation in \eqref{scale-separation}, we compute by only taking into account self-interactions
\[
\begin{split}
&\quad\int_0^t P_{< N_{1,k+1}}\mathbb P\div( v_{k+1}\otimes v_{k+1}-\nabla^\perp h_{k+1}\otimes \nabla^\perp h_{k+1}) \, ds \\
&\approx P_{< N_{1, k+1}}\mathbb P\div \sum_{j\in \Lambda_u} \mathbb P\Delta \psi_{j,k+1,u} \otimes \mathbb P\Delta \psi_{j,k+1,u}\\
&\quad+P_{< N_{1, k+1}}\mathbb P\div \sum_{j\in \Lambda_b} \left(\mathbb P\Delta \psi_{j,k+1,c} \otimes \mathbb P\Delta \psi_{j,k+1,c}\right)\\
&\quad-P_{< N_{1, k+1}}\mathbb P\div \sum_{j\in \Lambda_b} \left((\tilde a_{j,k+1,b}\eta^{j,\perp}\sin(N_{j,k+1}\eta^j\cdot x))\otimes (\tilde a_{j,k+1,b}\eta^{j,\perp}\sin(N_{j,k+1}\eta^j\cdot x))\right)\\
& \approx \mathbb P\div\sum_{j\in\Lambda_u}a_{j,k+1,u}^2\eta^{j,\perp}\otimes\eta^{j,\perp}+\mathbb P\div\sum_{j\in\Lambda_b}\left(a_{j,k+1,b}^2-\tilde a_{j,k+1,b}^2\right)\eta^{j,\perp}\otimes\eta^{j,\perp}
\end{split}
\]
and 
\[
\begin{split}
&\quad-\int_0^t P_{< N_{1,k+1}}\div( v_{k+1} h_{k+1}) \, ds \\
&\approx -\int_0^t P_{< N_{1,k+1}}( v_{k+1}\cdot \nabla h_{k+1})  \, ds 
\end{split}
\]
for $t\geq t_k$.

The break point is that the leading order term in $v_{k+1}\cdot \nabla h_{k+1}$ vanishes; otherwise, one would be able to use a variation of Lemma \ref{le-geometry-TV} to carry through the iterative process applying convex integration as done in the previous section. 
The vanishing leading order term comes exactly from the geometry of the induction equation in the genuine 2D case. 
This is consistent with what is found in convex integration study for 2D MHD \cite{FaracoLindbergSzekelyhidi2021}.

The fact of the leading order term being vanished in the nonlinear interaction in the magnetic equation seems to indicate that the geometry tends to prevent blowup of the magnetic field. This suggests us to construct solutions $(u,B)$ for the 2D MHD \eqref{mhd-2d} with nontrivial but controlled magnetic field and instantaneous blowup velocity field as stated in Theorem \ref{thm-2dmhd}.



\addtocontents{toc}{\protect\setcounter{tocdepth}{2}}

\section{Building blocks}\label{sec:principal}

This section fixes the frequency hierarchy, the pipe geometry, and the potential profiles used later to build the principal part of the solution. The spatial localization scheme is inherited from \cites{CDP,CoiculescuPalasek2025}; the point here is to record the 2D version we need and to prepare the coupled amplitude construction for the passive scalar.

\subsection{The basic building blocks}\label{frequency_scales_section}


Fix two growth parameters $A>1$ and $b>1$ sufficiently large. We also choose an integer $m_*\in\mathbb N$ later. We normalize the first frequency by $N_{1,0}=1$ and define all remaining scales by
\begin{equation}\label{N_definition}
N_{j,k}=m_*\big\lceil A^{b^{k+(j-1)/J}}\big\rceil
\end{equation}
for $(j,k)\neq(1,0)$, where $\lceil \cdot\rceil$ denotes the ceiling function. If $A$ is large enough in terms of $b$, $d$, and $m_*$, then these scales are strictly ordered in the sense that
\[
N_{j_1,k_1}\ll N_{j_2,k_2}
\]
whenever either $k_1<k_2$, or $k_1=k_2$ and $j_1<j_2$.

We next introduce intermediate scales, which will govern the geometry of the cutoffs. For $k\geq1$ and $j\in\mathcal J$, set
\begin{equation}\label{M_definition}
    M_{j,k}=\left\{
\begin{aligned}
&\big\lceil A^{\gamma b^{k}}\big\rceil, \quad &j=1,\\
&\big\lceil A^{\gamma b^{(j-1)/J}}\big\rceil M_{1,k}, \quad &j\geq2
\end{aligned}
\right.
\end{equation}
where $\gamma\in(0,1)$ will be fixed later. If, in addition,
\begin{align*}
    \gamma> b^{-1/J},
\end{align*}
and $A$ is taken larger if necessary, then we may arrange
\begin{equation}\begin{aligned}\label{N_and_M_ordering}
    A^cN_{j-1,k}\leq M_{j,k}\leq A^{-c}N_{j,k}&,\quad 2\leq j\leq J,\\
    A^cN_{J,k-1}\leq M_{1,k}\leq A^{-c} N_{1,k}&
\end{aligned}\end{equation}
for some $c>0$ (depending on $b$ and $\gamma$).

The geometric arrangement of the building blocks follows the periodic pipe construction from \cites{CDP,CoiculescuPalasek2025}. Choose rational directions $\eta^j\in\mathbb S^1\cap\mathbb Q^2$, $j\in\mathcal J$, so that $\{\eta^j,\eta^{j,\perp}\}$ is an orthonormal basis for each $j$. We now fix $m_*$ so that $m_*\eta^j\in\mathbb Z^2$ for all $j$, which guarantees $N_{j,k}\eta^j\in\mathbb Z^2$ at every scale.

For $\rho<1/10$, let $\mathcal C_{j,k}(\rho)$ denote the $2\pi/M_{j,k}$-periodic cylinder of radius $\rho M_{j,k}^{-1}$ around the periodic line $\mathbb R\eta^j\pmod{2\pi\mathbb Z^2/M_{j,k}}$. Choose $\delta_0>0$ small enough so that
\begin{align}\label{pipe_volume_bound}
|\mathcal C_{j,k}(4\delta_0)|\leq1/(10J).
\end{align}
For each $k\geq1$, choose a $2\pi/M_{j,k}$-periodic cutoff $\varphi_{j,k}\in C_c^\infty(\mathcal C_{j,k}(\delta_0))$ satisfying $\eta^{j,\perp}\cdot\nabla\varphi_{j,k}=0$ and normalized by
\[
(2\pi)^{-2}\int_{\mathbb T^2}\varphi_{j,k}^2(x)\sin^2(N_{j,k}\eta^j\cdot x) \, dx =1, \qquad \forall j,k,
\]
and with
\begin{equation} \label{eq:varphi_bound}
\|\nabla^n \varphi_{j,k}\|_\infty\lesssim M_{j,k}^{n}, \qquad \forall j,k.
\end{equation}
The active regions are then defined by
\begin{align*}
    \Omega_k=\bigcap_{k'=1}^k\bigcup_{j\in\mathcal J}\mathcal C_{j,k'}((3-2^{-(k-k')})\delta_0),\quad \widetilde\Omega_k=\bigcap_{k'=1}^k\bigcup_{j\in \mathcal J}\mathcal C_{j,k'}((3-\frac342^{-(k-k')})\delta_0),
\end{align*}
with the convention $\Omega_0=\widetilde\Omega_0=\mathbb T^2$.

\begin{lemma}[\cite{CDP}]
For any $k\geq0$, we have
    \begin{align}
    |\Omega_{k}|\leq2^{-k}|\mathbb T^2|.\label{Omega_volume_estimate}
\end{align}
Moreover, there is a constant $C_0>1$ such that if $Q\subset\mathbb T^2$ is a cube of length $\ell(Q)\in [C_0M_{1,k_0}^{-1},2\pi)$, then
\begin{align}\label{cube_intersection_volume_bound}
    |\Omega_{k}\cap Q|\leq2^{-(k-k_0)}|Q|\qquad\forall k\geq k_0.
\end{align}
\end{lemma}

Finally, let $\chi_k\in C_c^\infty(\widetilde\Omega_{k-1})$ be cutoffs with $\chi_k\equiv1$ on $\Omega_{k-1}$ and
\begin{align} \label{eq:chi_k_bound}
    \|\nabla^n\chi_k\|_\infty\lesssim_nM_{J,k-1}^n.
\end{align}

\subsection{The velocity and scalar potentials}\label{sec:potential}

Let $\phi_k$ be a standard mollifier at scale
\[
\ell_k:=N_{1,k}^{-\frac12}N_{1,k+1}^{-\frac12}.
\]
The choice of $b$ guarantees $\ell_k\leq N_{J,k}^{-1}$. We now introduce three families of potential profiles, corresponding to the velocity piece, the passive scalar piece, and the coupling term:
\begin{align}
\psi_{j,0,u}(x)&\coloneqq N_{j,0}^{-2}a_{j,0,u}(x)\eta^{j,\perp}\sin(N_{j,0}\eta^j\cdot x), \qquad \eta_j\in \Lambda_u, \label{def_psi_0}\\
\psi_{j,k,u}(x)&\coloneqq N_{j,k}^{-2}\phi_k*(a_{j,k,u}(x)\varphi_{j,k}(x)\eta^{j,\perp}\sin(N_{j,k}\eta^j\cdot x)),\quad \eta_j\in \Lambda_u, k\geq1, \label{def_psi}
\end{align}
\begin{align}
\psi_{j,0,b}(x)&\coloneqq N_{j,0}^{-2}\sin(N_{j,0}\eta^j\cdot x), \qquad \eta_j\in \Lambda_b, \label{def_psi_0b}\\
\psi_{j,k,b}(x)&\coloneqq N_{j,k}^{-2}\phi_k*(\varphi_{j,k}(x)\sin(N_{j,k}\eta^j\cdot x)),\quad \eta_j\in \Lambda_b, k\geq1, \label{def_psi_b}
\end{align}
\begin{align}
\psi_{j,0,c}(x)&\coloneqq N_{j,0}^{-2}a_{j,0,b}(x)\eta^{j,\perp}\sin(N_{j,0}\eta^j\cdot x), \qquad \eta_j\in \Lambda_b,\label{def_psi_0c}\\
\psi_{j,k,c}(x)&\coloneqq N_{j,k}^{-2}\phi_k*(a_{j,k,b}(x)\varphi_{j,k}(x)\eta^{j,\perp}\sin(N_{j,k}\eta^j\cdot x)),\quad \eta_j\in \Lambda_b, k\geq1,\label{def_psi_c}
\end{align}
The amplitudes $a_{j,k,u}$ and $a_{j,k,b}$ are chosen inductively so that
\begin{equation}\label{a-convex-integration-identity}
\begin{split}
\sum_{j\in \Lambda_u}a_{j,k+1,u}^2\eta^{j,\perp}\otimes\eta^{j,\perp} &=2 \mathcal D \sum_{j\in\Lambda_u} N_{j,k}\psi_{j,k,u}+p\Id, \\
\sum_{j\in\Lambda_b}a_{j,k+1,b}^2\eta^{j,\perp}\otimes\eta^{j,\perp}&=2 \mathcal D \sum_{j\in\Lambda_b}  N_{j,k}\psi_{j,k,c}+p\Id,\\
\sum_{j\in\Lambda_b}a_{j,k+1,b}\eta^{j,\perp} &=2\nabla \sum_{j\in\Lambda_b}  N_{j,k}\psi_{j,k,b}
\end{split}
\end{equation}
for some scalar function $p(x)$,
\begin{equation} \label{eq:a_bounds}
    \|\nabla^na_{j,k,u}\|_\infty+  \|\nabla^na_{j,k,b}\|_\infty\lesssim N_{J,k-1}^{n}, \qquad \forall j,k,
\end{equation}
and
\begin{equation}\label{a-support}
    \supp a_{j,k,u}\subset \widetilde\Omega_{k-1}, \qquad \supp a_{j,k,b}\subset \widetilde\Omega_{k-1}.
\end{equation}
The next lemma packages the properties of these amplitudes.

To streamline notation, we will also write
\[a_{j,k}=a_{j,k,u} \qquad \mbox{for}\quad \eta^j\in \Lambda_u; \qquad a_{j,k}=a_{j,k,b} \qquad \mbox{for}\quad \eta^j\in \Lambda_b\]
and likewise
\[\psi_{j,k}=\psi_{j,k,u} \qquad \mbox{for}\quad \eta^j\in \Lambda_u; \qquad \psi_{j,k}=\psi_{j,k,b} \qquad \mbox{for}\quad \eta^j\in \Lambda_b,\]
while $\psi_{j,k}=\psi_{j,k,c}$ when we refer to the coupling term.

\begin{lemma}
    For $j\in\mathcal J$ and $k\geq0$, there exist coefficient functions $a_{j,k,u}, a_{j,k,b}\in C^\infty(\mathbb T^2;\mathbb R)$ satisfying \eqref{a-convex-integration-identity}, \eqref{eq:a_bounds} and \eqref{a-support}. Further, $\supp\psi_{j,k}\subset \Omega_k$ and
    \begin{align}\label{eq:Dpsi_bounds}
        \|\nabla^n\psi_{j,k}\|_{L^\infty}\lesssim_n N_{j,k}^{-2+n},
    \end{align}
    for $n\geq0$.
\end{lemma}

\begin{proof}
    We construct the $u$-amplitudes and the $b$-amplitudes separately.

    \smallskip
    \noindent\textit{Step 1: the $u$-amplitudes.}
    We follow Lemma 4.2 of \cite{CDP} and only record the pieces needed later in the paper.

    Set
    \[
    a_{1,0,u}(x)=1,\qquad a_{j,0,u}(x)=0 \quad \text{for } j\neq1,
    \]
    and for $k\geq0$ define recursively
    \[
    a_{j,k+1,u}(x)=c^{-\frac12}\chi_{k+1}(x)\Gamma_j \big(\Id+cS_{k,u}(x)\big),
    \qquad
    S_{k,u}\coloneqq2\mathcal D \sum_{j\in\Lambda_u} N_{j,k}\psi_{j,k,u}.
    \]
    Here $\psi_{j,k,u}$ is given by \eqref{def_psi_0}--\eqref{def_psi}. Choosing the absolute constant $c>0$ small enough so that
    \begin{equation} \label{eq:condition_on_c}
    c\|S_{k,u}\|_\infty \leq c_0 \qquad \forall k
    \end{equation}
    keeps the argument inside the range of Lemma \ref{le-geometry}. The proof of Lemma 4.2 in \cite{CDP} then gives the bounds \eqref{eq:a_bounds} and \eqref{eq:Dpsi_bounds} for $a_{j,k,u}$ and $\psi_{j,k,u}$.

    Applying Lemma \ref{le-geometry} yields
    \begin{align*}
        \sum_{j\in\Lambda_u} a_{j,k+1,u}^2\eta^{j,\perp}\otimes\eta^{j,\perp}
        &=c^{-1}\chi_{k+1}^2(x) \sum_{j\in\Lambda_u} \Gamma_j^2\big(\Id+cS_{k,u}(x)\big)\eta^{j,\perp}\otimes\eta^{j,\perp}\\
        &=\chi_{k+1}^2(x)(c^{-1}\Id+S_{k,u}(x)).
    \end{align*}
    Thus the first identity in \eqref{a-convex-integration-identity} will follow once we know that $\chi_{k+1}\equiv1$ on $\supp S_{k,u}$. For this, observe that
    \[
    \supp\psi_{j,k,u}\subset (\supp a_{j,k,u}\cap\supp\varphi_{j,k})+B(0,\ell_k),
    \]
    where the sum is in the Minkowski sense. Using \eqref{a-support}, the inclusion $\supp\varphi_{j,k}\subset\mathcal C_{j,k}(\delta_0)$, and the definitions of $\Omega_k$ and $\widetilde\Omega_k$, we obtain
    \[
    \supp\psi_{j,k,u}\subset \Omega_k
    \]
    provided $\ell_kM_{j,k'}\ll2^{-(k-k')}$ for every $j$ and $k'\leq k$. Hence $\supp S_{k,u}\subset\Omega_k$, so $\chi_{k+1}\equiv1$ on $\supp S_{k,u}$, and the first identity follows. The factor $\chi_{k+1}$ also gives \eqref{a-support} for $a_{j,k+1,u}$.

    \smallskip
    \noindent\textit{Step 2: the coupled $b$-amplitudes.}
    This is the new part, where the tensor and vector constraints must be realized by the same coefficients.

    Set
    \[
    a_{j,0,b}\equiv 0,\qquad j\in\Lambda_b,
    \]
    and for $k\ge0$ define
    \[
    S_{k,c}\coloneqq 2\mathcal D\sum_{j\in\Lambda_b}N_{j,k}\psi_{j,k,c},\qquad
    S_{k,b}\coloneqq 2\nabla\sum_{j\in\Lambda_b}N_{j,k}\psi_{j,k,b}.
    \]
    Let $\Gamma_j(R,g)>0$ and $p(R,g)$ be the coefficient maps from Lemma \ref{le-geometry-TV}. We then set
    \[
    a_{j,k+1,b}(x)\coloneqq c^{-1/2}\chi_{k+1}(x)\Gamma_j\!\Big(cS_{k,c}(x),\,c^{\frac12}S_{k,b}(x)\Big).
    \]

    From the definitions \eqref{def_psi_0b}--\eqref{def_psi_c} we first obtain the rough bounds
    \[
    \|\nabla^n\psi_{j,k,b}\|_\infty+\|\nabla^n\psi_{j,k,c}\|_\infty
    \lesssim N_{j,k}^{-2}\ell_k^{-n}c^{-1/2},
    \]
    and therefore
    \[
    \|\nabla^nS_{k,c}\|_\infty+\|\nabla^nS_{k,b}\|_\infty
    \lesssim N_{1,k}^{-1}\ell_k^{-n-1}c^{-1/2}.
    \]
    Choosing $c>0$ sufficiently small keeps the coupled input admissible. Applying the multivariate Fa\`a di Bruno formula to the composed coefficient maps gives
    \[
    \|\nabla^n a_{j,k+1,b}\|_\infty\lesssim c^{\frac{n-1}{2}}N_{1,k}^{-n}\ell_k^{-2n}.
    \]
    As in the $u$-part, this rough estimate can then be bootstrapped, using $\ell_{k-1}^{-1}=N_{1,k-1}^{1/2}N_{1,k}^{1/2}$ and $M_{j,k}\le N_{j,k}$, to the sharper bounds
    \[
    \|\nabla^n\psi_{j,k,b}\|_\infty+\|\nabla^n\psi_{j,k,c}\|_\infty
    \lesssim N_{j,k}^{-2}c^{-1/2}N_{j,k}^{n},
    \]
    \[
    \|\nabla^nS_{k,c}\|_\infty+\|\nabla^nS_{k,b}\|_\infty
    \lesssim c^{-1/2}N_{J,k}^{n},
    \]
    and hence
    \[
    \|\nabla^na_{j,k+1,b}\|_\infty\lesssim_n N_{J,k}^{n}.
    \]
    This gives \eqref{eq:a_bounds} for the $b$-amplitudes as well.

    Applying Lemma \ref{le-geometry-TV} at
    \[
    (R,g)=\big(cS_{k,c}(x),\,c^{\frac12}S_{k,b}(x)\big),
    \]
    we obtain
    \begin{align*}
    \sum_{j\in\Lambda_b}a_{j,k+1,b}^2\eta^{j,\perp}\otimes\eta^{j,\perp}
    &=\chi_{k+1}^2\Big(S_{k,c}-c^{-1} p(cS_{k,c},c^{\frac12}S_{k,b})\Id\Big),\\
    \sum_{j\in\Lambda_b}a_{j,k+1,b}\eta^{j,\perp}
    &=\chi_{k+1}S_{k,b}.
    \end{align*}
    The same support argument as above gives
    \[
    \supp\psi_{j,k,b}\subset\Omega_k,\qquad \supp\psi_{j,k,c}\subset\Omega_k,
    \]
    hence $\supp S_{k,c},\supp S_{k,b}\subset\Omega_k$. Therefore $\chi_{k+1}\equiv1$ on these supports, so
    \[
    \chi_{k+1}^2S_{k,c}=S_{k,c},\qquad
    \chi_{k+1}S_{k,b}=S_{k,b},
    \]
    and the second and third identities in \eqref{a-convex-integration-identity} follow, with the scalar term absorbed into the pressure.

    Finally, \eqref{a-support} for $a_{j,k,b}$ is immediate from the factor $\chi_{k+1}$ in its definition, and the support inclusions above show
    \[
    \supp\psi_{j,k,u},\supp\psi_{j,k,b},\supp\psi_{j,k,c}\subset\Omega_k.
    \]
\end{proof}

\section{The principal part of the solution}\label{sec:est}

\subsection{Definition of the principal fields}
We now pass from the building blocks of Section~4 to the fields that drive the
principal cascade. We begin with the velocity and passive-scalar pieces,
\begin{equation}\label{app-prin}
\begin{split}
\bar v_k(x,t)&= \sum_{j\in\Lambda_u} \bar v_{j,k}(x,t) +\sum_{j\in\Lambda_b} \bar v_{j,k}(x,t)\\
&= \sum_{j\in\Lambda_u} -N_{j,k} e^{-N_{j,k}
^2t} \Delta \psi_{j,k,u}(x)+\sum_{j\in\Lambda_b} -N_{j,k} e^{-N_{j,k}
^2t} \Delta \psi_{j,k,c}(x),\\
\bar h_k(x,t)&= \sum_{j\in\Lambda_b} \bar h_{j,k}(x,t) = \sum_{j\in\Lambda_b} -N_{j,k} e^{-N_{j,k}
^2t} \Delta \psi_{j,k,b}(x)
\end{split}
\end{equation}
for each $k\geq0$. The corresponding principal contributions are obtained by
feeding the next frequency level into the Duhamel formulas
\begin{equation}\label{prin}
\begin{split}
v_k(x,t) &= -\int_0^{t} e^{(t-s)\Delta }\mathbb P\div (\bar v_{k+1}\otimes \bar v_{k+1})(s) \, ds,\\
h_k(x,t) &= -\int_0^{t} e^{(t-s)\Delta }\div (\bar v_{k+1} \bar h_{k+1})(s) \, ds
\end{split}
\end{equation}
for $k\geq 0$.

Summing over all levels gives the full principal pair
\begin{align}\label{full-prin}
    v(x,t)=\sum_{k=0}^\infty v_k(x,t), \qquad h(x,t)=\sum_{k=0}^\infty h_k(x,t).
\end{align}
We likewise set
\begin{align}\label{full-app}
    \bar v(x,t)=\sum_{k=0}^\infty \bar v_k(x,t), \qquad  \bar h(x,t)=\sum_{k=0}^\infty \bar h_k(x,t).
\end{align}

For later estimates it is convenient to keep track of inverse-divergence
primitives. We therefore introduce tensor fields $R_k$, $\bar R_k$, and vectors $H_k$,
$\bar H_k$ through
\begin{equation*}
v_k=\div R_k, \qquad \bar v_k=\div\bar R_k, \qquad h_k=\div H_k, \qquad \bar h_k=\div\bar H_k.
\end{equation*}
More explicitly, we set
\begin{equation}\label{tensor-app}
\begin{split}
\bar R_k(x,t)
& = \sum_{j\in\Lambda_u} -N_{j,k} e^{-N_{j,k}^2t}  \newD  \psi_{j,k,u}(x)+\sum_{j\in\Lambda_b} -N_{j,k} e^{-N_{j,k}^2t}  \newD  \psi_{j,k,c}(x),\\
&=: \bar R_{k,u}(x,t)+ \bar R_{k,c}(x,t),\\
\bar H_k(x,t)&= \sum_{j\in\Lambda_b} \bar H_{j,k}(x,t)
 = \sum_{j\in\Lambda_b} -N_{j,k} e^{-N_{j,k}^2t}  \nabla  \psi_{j,k,b}(x)
\end{split}
\end{equation}
and
\begin{equation}\label{tensor-RH}
\begin{split}
R_k(x,t) &= -\int_0^{t} e^{(t-s)\Delta }\mathcal R\mathbb P\div (\bar v_{k+1}\otimes \bar v_{k+1})(s) \, ds,\\
H_k(x,t) &= -\int_0^{t} e^{(t-s)\Delta }\mathcal R_1\div (\bar v_{k+1} \bar h_{k+1})(s) \, ds
\end{split}
\end{equation}
where $\mathcal R$ is the inverse-divergence operator from Subsection
\ref{sec:notation} and
\[
\mathcal R_1=\Delta^{-1}\nabla.
\]

The next subsection compares $(v,h)$ with its approximation
$(\bar v,\bar h)$. That comparison shows that the principal pair solves a
forced system whose error terms are small enough for the later perturbative
argument.

\subsection{Error estimates}

We begin by estimating the tensor fields $(\bar R_k,\bar H_k)$. The
argument separates into an early-time and a late-time regime, with the dividing
scale
\begin{equation}\label{time-sequence}
N_{1, k+1}^{-2} \ll t_k=N_{J_d,k}^{-4} \ll N_{J,k}^{-3}.
\end{equation}
\begin{proposition}\label{barv-I_estimate_proposition}
Denote
\[
\begin{split}
\mathcal{I}_{k,u} &= -\int_0^{t} e^{(t-s)\Delta }\sum_{j\in \Lambda_u} N_{j,k+1}^2 e^{-2N_{j,k+1}^2s} \mathbb Q (a_{j,k+1,u}^2 \eta^{j,\perp}\otimes \eta^{j,\perp}) \, ds,\\
\mathcal{I}_{k,c} &= -\int_0^{t} e^{(t-s)\Delta }\sum_{j\in \Lambda_b} N_{j,k+1}^2 e^{-2N_{j,k+1}^2s} \mathbb Q \!\left(a_{j,k+1,b}^2 (\eta^{j,\perp}\otimes \eta^{j,\perp})\right) \, ds,\\
\mathcal{I}_{k,b} &= -\int_0^{t} e^{(t-s)\Delta }\sum_{j\in \Lambda_b} N_{j,k+1}^2 e^{-2N_{j,k+1}^2s} \mathcal R_1\div \!\left(a_{j,k+1,b} \eta^{j,\perp}\right) \, ds.
\end{split}
\]
Then for any $\varepsilon_0 >0$, $\alpha \in(0, \frac{1}{10})$, and $\bar n \in \mathbb{N}$, the following bounds hold for all $t\geq 0$:
\[
\|\nabla^n(\bar R_k(t) - \mathcal{I}_{k,u}(t)- \mathcal{I}_{k,c}(t))\|_{L^\infty} \leq \varepsilon_0 N^{-\alpha}_{1,k}(t^{-\frac{n}{2} +\alpha}+1), \qquad  n=1,2,\dots,\bar n,
\]
\[
\|\bar R_k(t) - \mathcal{I}_{k,u}(t)- \mathcal{I}_{k,c}(t)\|_{L^\infty} \lesssim_\varepsilon \varepsilon_0N_{1,k}^{-\alpha}+\mathbbm{1}_{t\leq t_k}N_{J, k}^{\varepsilon}, 
\]
\[
\|\nabla^n(\bar H_k(t) - \mathcal{I}_{k,b}(t))\|_{L^\infty} \leq \varepsilon_0 N^{-\alpha}_{1,k}(t^{-\frac{n}{2} +\alpha}+1), \qquad  n=1,2,\dots,\bar n,
\]
\[
\|\bar H_k(t) - \mathcal{I}_{k,b}(t)\|_{L^\infty} \lesssim_\varepsilon \varepsilon_0N_{1,k}^{-\alpha}+\mathbbm{1}_{t\leq t_k}N_{J, k}^{\varepsilon},
\]
provided $A$ and $b$ are large enough.
\end{proposition}
\begin{proof}
The heat-semigroup decomposition used below is the same one that appears in
\cites{Dai2026,CDP}. We spell out Step I once and then record only the
passive-scalar-specific substitutions in Steps II and III.
{\textbf{Step I.}}
This step is to show 
\begin{equation}\label{step1}
\begin{split}
\|\nabla^n(\bar R_{k,u}(t) - \mathcal{I}_{k,u}(t))\|_{L^\infty} &\leq \varepsilon_0 N^{-\alpha}_{1,k}(t^{-\frac{n}{2} +\alpha}+1), \qquad  n=1,2,\dots,\bar n,\\
\|\bar R_{k,u}(t) - \mathcal{I}_{k,u}(t)\|_{L^\infty} &\lesssim_\varepsilon \varepsilon_0N_{1,k}^{-\alpha}+\mathbbm{1}_{t\leq t_k}N_{J, k}^{\varepsilon}.
\end{split}
\end{equation}
This is the same oscillatory stress estimate proved in the corresponding
velocity step of \cites{Dai2026,CDP}; here we only retain the local formulas
that will be reused in Steps II and III.

As in \cites{Dai2026,CDP}, for any $\varepsilon>0$,
\[
\|\nabla^n\mathcal{I}_{k,u}(t)\|_\infty \lesssim_\varepsilon N_{J,k}^{n+\varepsilon},
\qquad
\|\nabla^n\bar R_{k,u}(t)\|_\infty\lesssim \sum_j N_{j,k}^{n}e^{-N_{j,k}^2t}\lesssim N_{J,k}^{n},
\]
and therefore, on $[0,t_k]$,
\begin{equation} \label{eq:bound_before_t_k_2D}
\|\nabla^n(\bar R_{k,u}(t)-\mathcal I_{k,u}(t))\|_\infty
\lesssim_\varepsilon N_{J,k}^{n+\varepsilon}.
\end{equation}

For $t\ge t_k$, we split
\[
\mathcal I_{k,u}(t)=\mathcal I_{k,u}^{<t_k}(t)+\mathcal I_{k,u}^{>t_k}(t)
\]
and define
\[
\tilde{\mathcal I}_{k,u}^{<t_k}
:=-e^{(t-t_k)\Delta}\int_0^{t_k}\sum_{j\in\Lambda_u}N_{j,k+1}^2e^{-2N_{j,k+1}^2s}
\mathbb Q\!\left(a_{j,k+1,u}^2\eta^{j,\perp}\otimes\eta^{j,\perp}\right)\,ds.
\]
The same semigroup truncation estimate as in \cites{Dai2026,CDP} yields
\[
\|\nabla^n(\mathcal I_{k,u}^{<t_k}-\tilde{\mathcal I}_{k,u}^{<t_k})\|_\infty
\lesssim t_kN_{J,k}^{n+2+\varepsilon}.
\]
Moreover, by the first identity in \eqref{a-convex-integration-identity} and
\eqref{eq:Q_identities},
\begin{equation}\label{eq:quadratic_identity}
\frac12\sum_{j\in\Lambda_u}\mathbb Q\!\left(a_{j,k+1,u}^2\eta^{j,\perp}\otimes\eta^{j,\perp}\right)
=\sum_{j\in\Lambda_u}N_{j,k}\Big(\newD\psi_{j,k,u}
+(\Id-2\Delta^{-1}\nabla\otimes\nabla)\div\psi_{j,k,u}\Big),
\end{equation}
so that
\[
\tilde{\mathcal I}_{k,u}^{<t_k}
=\bar R_{k,u}+\tilde{\mathcal I}_{k,u}^1+\tilde{\mathcal I}_{k,u}^2+\tilde{\mathcal I}_{k,u}^3
-(\Id-2\frac{\nabla\otimes\nabla}{\Delta})p,
\]
where
\[
\begin{split}
\tilde{\mathcal I}_{k,u}^1
&:=-\sum_{j\in\Lambda_u}N_{j,k}\big(e^{-N_{j,k}^2(t-t_k)}-e^{-N_{j,k}^2t}\big)\newD\psi_{j,k,u},\\
\tilde{\mathcal I}_{k,u}^2
&:=\sum_{j\in\Lambda_u}N_{j,k}\big(e^{-N_{j,k}^2(t-t_k)}-e^{(t-t_k)\Delta}\big)\newD\psi_{j,k,u},\\
\tilde{\mathcal I}_{k,u}^3
&:=\frac12e^{(t-t_k)\Delta}\sum_{j\in\Lambda_u}e^{-2N_{j,k+1}^2t_k}
\mathbb Q\!\left(a_{j,k+1,u}^2\eta^{j,\perp}\otimes\eta^{j,\perp}\right).
\end{split}
\]
Exactly as in the companion proofs, Lemma~\ref{l:commutator},
\eqref{heat-decay-estimate}, and the late-time integral estimate imply, for
$t\ge t_k$,
\[
\begin{split}
\|\nabla^n(\bar R_{k,u}-\mathcal I_{k,u})\|_\infty
&\lesssim_{\varepsilon,m}\sum_{j\in\Lambda_u}\Big(
N_{j,k}^{n}e^{-N_{j,k}^2t/4}(N_{j,k}^2t_k+M_{j,k}N_{j,k}^{-1})
+N_{j,k}^{n-m}M_{j,k}^m\Big)\\
&\quad +\sum_{j\in\Lambda_u}\Big(
M_{j,k}N_{j,k}^{n-1+\varepsilon}e^{-N_{j,k}^2(t-t_k)/4}
+N_{j,k}^{n-m+\varepsilon}M_{j,k}^m\Big)\\
&\quad +\sum_{j\in\Lambda_u}N_{J,k}^{n+\varepsilon}e^{-2N_{j,k+1}^2t_k}
+t_kN_{J,k}^{n+2+\varepsilon}\\
&\lesssim N_{j,k}^{n-\frac12}e^{-N_{j,k}^2t/4}
\end{split}
\]
where the last step is achieved by
choosing $\gamma=\frac12$ in \eqref{M_definition}
and taking
$m$, $A$, and $b$ sufficiently large. Combining this with
\eqref{eq:bound_before_t_k_2D} on $[0,t_k]$, we obtain
\[
\|\nabla^n(\bar R_{k,u}(t)-\mathcal I_{k,u}(t))\|_\infty
\le_n \varepsilon_0N_{1,k}^{-\alpha}(t^{-\frac n2+\alpha}+1),
\qquad n=1,2,\dots,\bar n.
\]
For $n=0$, the corresponding late-time estimate becomes
\[
\|\bar R_{k,u}(t)-\mathcal I_{k,u}(t)\|_\infty
\lesssim_\varepsilon \varepsilon_0N_{1,k}^{-\alpha},
\qquad t\ge t_k,
\]
while \eqref{eq:bound_before_t_k_2D} supplies the
$\mathbbm{1}_{t\le t_k}N_{J,k}^\varepsilon$ term on $[0,t_k]$. This establishes
\eqref{step1}.

{\textbf{Step II.}}
We next verify
\begin{equation}\label{step2}
\begin{split}
\|\nabla^n(\bar R_{k,c}(t) - \mathcal{I}_{k,c}(t))\|_{L^\infty} &\leq \varepsilon_0 N^{-\alpha}_{1,k}(t^{-\frac{n}{2} +\alpha}+1), \qquad  n=1,2,\dots,\bar n,\\
\|\bar R_{k,c}(t) - \mathcal{I}_{k,c}(t)\|_{L^\infty} &\lesssim_\varepsilon \varepsilon_0N_{1,k}^{-\alpha}+\mathbbm{1}_{t\leq t_k}N_{J, k}^{\varepsilon}.
\end{split}
\end{equation}
The proof follows the same pattern as Step I, together with the analogous
coupled-stress estimates in \cites{Dai2026,CDP}. The required substitutions are
\[
(\Lambda_u,a_{j,k+1,u},\psi_{j,k,u},\bar R_{k,u},\mathcal I_{k,u},p)
\]
replaced by
\[
(\Lambda_b,a_{j,k+1,b},\psi_{j,k,c},\bar R_{k,c},\mathcal I_{k,c},p_c),
\]
and the second identity in \eqref{a-convex-integration-identity} is used in
place of the first.

The same preliminary estimate as in Step I gives, for $t\in[0,t_k]$,
\begin{equation}\label{eq:bound_before_t_k_2D_c}
\|\nabla^n(\bar R_{k,c}(t)-\mathcal I_{k,c}(t))\|_\infty
\lesssim_\varepsilon N_{J_d,k}^{n+\varepsilon}.
\end{equation}
For $t\ge t_k$, we decompose
\[
\mathcal I_{k,c}(t)=\mathcal I_{k,c}^{<t_k}(t)+\mathcal I_{k,c}^{>t_k}(t)
\]
and introduce
\[
\tilde{\mathcal I}_{k,c}^{<t_k}
:=-e^{(t-t_k)\Delta}\int_0^{t_k}\sum_{j\in\Lambda_b}N_{j,k+1}^2e^{-2N_{j,k+1}^2s}
\mathbb Q\!\left(a_{j,k+1,b}^2\eta^{j,\perp}\otimes\eta^{j,\perp}\right)\,ds.
\]
Exactly as in Step I, we have
\[
\|\nabla^n(\mathcal I_{k,c}^{<t_k}-\tilde{\mathcal I}_{k,c}^{<t_k})\|_\infty
\lesssim t_kN_{J_d,k}^{n+2+\varepsilon}.
\]
Moreover, combining the second identity in
\eqref{a-convex-integration-identity} with \eqref{eq:Q_identities} gives
\begin{equation}\label{eq:quadratic_identity_c}
\frac12\sum_{j\in\Lambda_b}\mathbb Q\!\left(a_{j,k+1,b}^2\eta^{j,\perp}\otimes\eta^{j,\perp}\right)
=\sum_{j\in\Lambda_b}N_{j,k}\Big(\newD\psi_{j,k,c}+(\Id-2\Delta^{-1}\nabla\otimes\nabla)\div\psi_{j,k,c}\Big),
\end{equation}
so that
\[
\tilde{\mathcal I}_{k,c}^{<t_k}
=\bar R_{k,c}+\tilde{\mathcal I}_{k,c}^1+\tilde{\mathcal I}_{k,c}^2+\tilde{\mathcal I}_{k,c}^3-(\Id-2\frac{\nabla\otimes\nabla}{\Delta})p_c,
\]
where $\tilde{\mathcal I}_{k,c}^1$, $\tilde{\mathcal I}_{k,c}^2$,
$\tilde{\mathcal I}_{k,c}^3$, and $p_c$ are defined exactly as in Step I after
making the substitutions above. Applying the same commutator estimate,
heat-decay bound, and late-time estimate, we obtain for $t\ge t_k$,
\[
\begin{split}
\|\nabla^n(\bar R_{k,c}-\mathcal I_{k,c})\|_\infty
&\lesssim_{\varepsilon,m}\sum_{j\in\Lambda_b}\Big(
N_{j,k}^{n}e^{-N_{j,k}^2t/4}(N_{j,k}^2t_k+M_{j,k}N_{j,k}^{-1})
+N_{j,k}^{n-m}M_{j,k}^m\Big)\\
&\quad +\sum_{j\in\Lambda_b}\Big(
M_{j,k}N_{j,k}^{n-1+\varepsilon}e^{-N_{j,k}^2(t-t_k)/4}
+N_{j,k}^{n-m+\varepsilon}M_{j,k}^m\Big)\\
&\quad +\sum_{j\in\Lambda_b}N_{J_d,k}^{n+\varepsilon}e^{-2N_{j,k+1}^2t_k}
+t_kN_{J_d,k}^{n+2+\varepsilon}.
\end{split}
\]
Choosing $\gamma=\frac12$ in \eqref{M_definition} yields
\[
N_{j,k}^{n}e^{-N_{j,k}^2t/4}(N_{j,k}^2t_k+M_{j,k}N_{j,k}^{-1})
\lesssim N_{j,k}^{n-\frac12}e^{-N_{j,k}^2t/4},
\]
and the remaining non-decaying terms are again made arbitrarily small by
taking $m$, $A$, and $b$ sufficiently large. Combining this with
\eqref{eq:bound_before_t_k_2D_c} on $[0,t_k]$ proves the bounds in
\eqref{step2}.

{\textbf{Step III.}}
We now prove
\begin{equation}\label{step3}
\begin{split}
\|\nabla^n(\bar H_k(t) - \mathcal{I}_{k,b}(t))\|_{L^\infty} &\leq \varepsilon_0 N^{-\alpha}_{1,k}(t^{-\frac{n}{2} +\alpha}+1), \qquad  n=1,2,\dots,\bar n,\\
\|\bar H_k(t) - \mathcal{I}_{k,b}(t)\|_{L^\infty} &\lesssim_\varepsilon \varepsilon_0N_{1,k}^{-\alpha}+\mathbbm{1}_{t\leq t_k}N_{J, k}^{\varepsilon}.
\end{split}
\end{equation}
This step has the same structure as Step I and the scalar-error estimate in
\cite{Dai2026}; we therefore record only the relevant substitutions. We replace
\[
(\Lambda_u,a_{j,k+1,u},\psi_{j,k,u},\bar R_{k,u},\mathbb Q,\newD)
\]
by
\[
(\Lambda_b,a_{j,k+1,b},\psi_{j,k,b},\bar H_k,\mathcal R_1\div,\nabla),
\]
and use the third identity in \eqref{a-convex-integration-identity}.

The corresponding rough estimate on $[0,t_k]$ is therefore
\begin{equation}\label{eq:bound_before_t_k_2D_B}
\|\nabla^n(\bar H_k(t)-\mathcal I_{k,b}(t))\|_\infty
\lesssim_\varepsilon N_{J_d,k}^{n+\varepsilon}.
\end{equation}
For $t\ge t_k$, we split
\[
\mathcal I_{k,b}(t)=\mathcal I_{k,b}^{<t_k}(t)+\mathcal I_{k,b}^{>t_k}(t)
\]
and set
\[
\tilde{\mathcal I}_{k,b}^{<t_k}
:=-e^{(t-t_k)\Delta}\int_0^{t_k}\sum_{j\in\Lambda_b}N_{j,k+1}^2e^{-2N_{j,k+1}^2s}
\mathcal R_1\div\!\left(a_{j,k+1,b}\eta^{j,\perp}\right)\,ds,
\]
Then
\[
\|\nabla^n(\mathcal I_{k,b}^{<t_k}-\tilde{\mathcal I}_{k,b}^{<t_k})\|_\infty
\lesssim t_kN_{J_d,k}^{n+2+\varepsilon}.
\]
Moreover, the third identity in \eqref{a-convex-integration-identity} gives
\begin{equation}\label{eq:quadratic_identity_B}
\frac12\sum_{j\in\Lambda_b}\mathcal R_1\div\!\left(a_{j,k+1,b}\eta^{j,\perp}\right)
=\sum_{j\in\Lambda_b}N_{j,k}\nabla\psi_{j,k,b},
\end{equation}
and hence
\[
\tilde{\mathcal I}_{k,b}^{<t_k}
=\bar H_k+\tilde{\mathcal I}_{k,b}^1+\tilde{\mathcal I}_{k,b}^2+\tilde{\mathcal I}_{k,b}^3,
\]
where $\tilde{\mathcal I}_{k,b}^1$, $\tilde{\mathcal I}_{k,b}^2$, and
$\tilde{\mathcal I}_{k,b}^3$ are the direct analogues of the Step I error
terms. Invoking the same commutator estimate and late-time bound then yields,
for
$t\ge t_k$,
\[
\begin{split}
\|\nabla^n(\bar H_k-\mathcal I_{k,b})\|_\infty
&\lesssim_{\varepsilon,m}\sum_{j\in\Lambda_b}\Big(
N_{j,k}^{n}e^{-N_{j,k}^2t/4}(N_{j,k}^2t_k+M_{j,k}N_{j,k}^{-1})
+N_{j,k}^{n-m}M_{j,k}^m\Big)\\
&\quad +\sum_{j\in\Lambda_b}\Big(
M_{j,k}N_{j,k}^{n-1+\varepsilon}e^{-N_{j,k}^2(t-t_k)/4}
+N_{j,k}^{n-m+\varepsilon}M_{j,k}^m\Big)\\
&\quad +\sum_{j\in\Lambda_b}N_{J_d,k}^{n+\varepsilon}e^{-2N_{j,k+1}^2t_k}
+t_kN_{J_d,k}^{n+2+\varepsilon}.
\end{split}
\]
With the same choice of parameters, every contribution is controlled exactly
as in Step I, and together with \eqref{eq:bound_before_t_k_2D_B} this proves
\eqref{step3}.

Finally, since $\bar R_k=\bar R_{k,u}+\bar R_{k,c}$, Step I and Step II imply
\[
\|\nabla^n(\bar R_k-\mathcal I_{k,u}-\mathcal I_{k,c})\|_\infty
\le \|\nabla^n(\bar R_{k,u}-\mathcal I_{k,u})\|_\infty+\|\nabla^n(\bar R_{k,c}-\mathcal I_{k,c})\|_\infty,
\]
with the analogous $n=0$ estimate. Combining this with Step III completes the
proof of Proposition~\ref{barv-I_estimate_proposition}.

\end{proof}

Applying Proposition \ref{barv-I_estimate_proposition}, we next show that
$R_k-\bar R_k$ is small. As a consequence, $v_k-\bar v_k$ is also small,
because $v_k=\div R_k$ and $\bar v_k=\div\bar R_k$.

\begin{proposition}\label{difference_estimate_proposition}
For all sufficiently small $\alpha>0$, all $\varepsilon_0>0$, and every
$\bar n \in \mathbb{N}$, the parameters $A$ and $b$ may be chosen large enough
so that
\[
\begin{split}
\|\nabla^n(R_k(t)- \bar R_k(t) )\|_{L^\infty} &\leq_n \varepsilon_0 N_{1,k}^{-\alpha} (t^{-\frac{n}{2}+\alpha}+1), \qquad n=1,2,\dots, \bar n,\\
\|R_k(t)- \bar R_k(t)\|_{L^\infty} &\lesssim_\varepsilon \varepsilon_0N_{1,k}^{-\alpha}+\mathbbm{1}_{t\leq t_k}N_{J, k}^{\varepsilon}.
\end{split}
\]

\end{proposition}

\begin{proof}
We use the same discrepancy decomposition as in the corresponding error
estimates from \cites{Dai2026,CDP}. More precisely, expanding
$\Delta\psi_{j,k+1,u}$ and $\Delta\psi_{j,k+1,c}$ separates each
unidirectional contribution into a principal part, an oscillatory correction,
and lower-order derivative or mollification remainders, while interactions
between different directions generate the same three classes of mixed terms as
in those references. Accordingly, we write
\[
R_k=\mathcal I_k+\mathcal J_k^1+\mathcal J_k^2+\mathcal E_k,
\]
where
\[
\begin{split}
\mathcal I_k
&:= -\int_0^{t} e^{(t-s)\Delta }\sum_{j\in\Lambda_u} N_{j,k+1}^2 e^{-2N_{j,k+1}^2s}\,
\mathbb Q\!\left(a_{j,k+1,u}^2 \eta^{j,\perp}\otimes \eta^{j,\perp}\right) \, ds\\
&\quad -\int_0^{t} e^{(t-s)\Delta }\sum_{j\in\Lambda_b} N_{j,k+1}^2 e^{-2N_{j,k+1}^2s}\,
\mathbb Q\!\left(a_{j,k+1,b}^2 \eta^{j,\perp}\otimes \eta^{j,\perp}\right) \, ds,
\end{split}
\]
\[
\begin{split}
\mathcal J_k^1
&:= -\int_0^{t} e^{(t-s)\Delta }\sum_{j\in\Lambda_u} N_{j,k+1}^2 e^{-2N_{j,k+1}^2s}\,
\mathbb Q\!\left(a_{j,k+1,u}^2(\varphi^2_{j,k+1}\sin^2(N_{j,k+1}\eta^j\!\cdot x)-1)\eta^{j,\perp}\otimes\eta^{j,\perp}\right)\,ds\\
&\quad -\int_0^{t} e^{(t-s)\Delta }\sum_{j\in\Lambda_b} N_{j,k+1}^2 e^{-2N_{j,k+1}^2s}\,
\mathbb Q\!\left(a_{j,k+1,b}^2(\varphi^2_{j,k+1}\sin^2(N_{j,k+1}\eta^j\!\cdot x)-1)\eta^{j,\perp}\otimes\eta^{j,\perp}\right)\,ds,
\end{split}
\]
\[
\begin{split}
\mathcal J_k^2
&:= -\int_0^t e^{(t-s)\Delta}\mathbb Q
\sum_{j,j'\in\Lambda_u,\ j\neq j'} N_{j,k+1}N_{j',k+1}e^{-N_{j,k+1}^2s-N_{j',k+1}^2s}\\
&\quad\qquad\qquad\cdot a_{j,k+1,u}a_{j',k+1,u}\varphi_{j,k+1}\varphi_{j',k+1}\sin(N_{j,k+1}\eta^j\!\cdot x)\sin(N_{j',k+1}\eta^{j'}\!\cdot x)\,\eta^{j,\perp}\otimes\eta^{j',\perp}ds\\
&\quad -\int_0^t e^{(t-s)\Delta}\mathbb Q
\sum_{j,j'\in\Lambda_b,\ j\neq j'} \mathrm{similar\ \ terms}\, ds\\
&\quad -\int_0^t e^{(t-s)\Delta}\mathbb Q
\sum_{j\in\Lambda_u, j'\in \Lambda_b} \mathrm{similar\ \ terms}\, ds
\end{split}
\]
and
\[
\mathcal E_k:=-\int_0^{t} e^{(t-s)\Delta }\mathbb Q(\mathfrak E_k^1+\mathfrak E_k^2)\,ds.
\]

Proposition~\ref{barv-I_estimate_proposition} already controls the leading term
$\mathcal I_k$. It remains to estimate $\mathcal J_k^1$, $\mathcal J_k^2$, and
$\mathcal E_k$.


\noindent
{\bf Bounds for $\mathcal{J}_k^1$ and $\mathcal{J}_k^2$.}

For $\mathcal J_k^1$, the same parallel-shear identity from
\cites{Dai2026,CDP} applies to both the $\Lambda_u$ and $\Lambda_b$
contributions,
since $\eta^{j,\perp}\cdot\eta^j=0$ and
\[
\varphi_{j,k+1}^2(x)\sin^2(N_{j,k+1}\eta^j\cdot x)-1
\]
has zero mean and is $2\pi/M_{j,k+1}$-periodic. Combining
Lemma~\ref{l:oscillation_estimate} with the same heat-semigroup estimates used
there gives
\[
\|\nabla^n \mathcal J_k^1\|_\infty
\lesssim \sum_{j}N_{J,k}^{1+\varepsilon}\big( N_{j,k+1}^{n-1} e^{-M_{j,k+1}^2t/4}+M_{j,k+1}^{-m+n}N_{J,k}^{m} \big).
\]
For $n=0$, the same argument gives
\[
\| \mathcal J_k^1 \|_\infty
\lesssim_\varepsilon \sum_{j}N_{J,k}^{1+\varepsilon}\big( M_{j,k+1}^{-1} e^{-M_{j,k+1}^2t/4}+M_{j,k+1}^{-m}N_{J,k}^{m} \big).
\]

For $\mathcal J_k^2$, the same commutator and heat-kernel argument from
\cites{Dai2026,CDP} applies to the $\Lambda_u$--$\Lambda_u$,
$\Lambda_b$--$\Lambda_b$, and mixed $\Lambda_u$--$\Lambda_b$ interactions.
Estimating $L^\infty$ through $C^\varepsilon$ and using that $\mathbb Q$ is a
Calder\'on--Zygmund operator, we obtain
\[
\|\nabla^n\mathcal J_k^2\|_\infty
\lesssim \sum_{j'<j} e^{-\frac14N_{j,k+1}^2t}N^{n-1+\varepsilon}_{j,k+1}N_{j',k+1}
\,+\sum_{j'<j} N_{j,k+1}^{-1}N_{j',k+1} N_{j,k+1}^{n-m}M_{j,k+1}^{m+\varepsilon}.
\]

\noindent
{\bf Bound for $\mathcal{E}_k$.}

Decompose $\mathcal E_k=\mathcal E_k^1+\mathcal E_k^2+\mathcal E_k^3$
according to the lower-order derivative contributions from the
unidirectional pieces, the lower-order interaction terms, and the
mollification remainders. The prototype expressions have the same structure as
in \cites{Dai2026,CDP}; the only modification is the replacement of
$a_{j,k+1,u}$ by $a_{j,k+1,b}$ in the $\Lambda_b$ and mixed components.
Consequently,
\[
\|\nabla^n \mathcal E^1_k\|_\infty\lesssim \sum_{j} \big(N_{j,k+1}^{n-1+\varepsilon}e^{-\frac14 N_{j,k+1}^2t} +N_{j,k+1}^{-2}M_{j,k+1}^{n+2+\varepsilon}\big),
\]
\[
\|\nabla^n \mathcal E_k^2\|_\infty\lesssim \|\nabla^n\mathcal J_k^2\|_\infty,
\]
and, using $\|\nabla^n(f-\phi_{k+1}*f)\|_\infty\lesssim \ell_{k+1}\|\nabla^{n+1}f\|_\infty$ together with \eqref{eq:a_bounds},
\[
\|\nabla^n\mathcal E_k^3\|_\infty\lesssim\sum_jN_{j,k+1}^{n+1+\varepsilon}\ell_{k+1}.
\]

\noindent
{\bf Conclusion.}
By the triangle inequality,
\[
\|\nabla^n(R_k-\bar R_k)\|_\infty \leq \|\nabla^n(R_k-\mathcal{I}_k)\|_\infty + \|\nabla^n(\bar R_k-\mathcal{I}_k)\|_\infty .
\]
Since $\mathcal I_k=\mathcal I_{k,u}+\mathcal I_{k,c}$, Proposition~\ref{barv-I_estimate_proposition} gives, for any $\varepsilon_0 >0$, $\alpha \in(0, \frac{1}{10})$, and $\bar n \in \mathbb{N}$,
\[
\|\nabla^n(\bar R_k - \mathcal{I}_k)\|_\infty \leq \varepsilon_0 N^{-\alpha}_{1,k+1}(t^{-\frac{n}{2} +\alpha}+1), \qquad n=1,2,\dots,\bar n,
\]
once $A$ and $b$ are taken sufficiently large.

Since
\[
R_k-\mathcal{I}_k=\mathcal{J}_k^1+\mathcal{J}_k^2 + \mathcal{E}_k,
\]
it remains to prove the same estimate for the residual terms
$\mathcal{J}_k^1+\mathcal{J}_k^2 + \mathcal{E}_k$. By the estimates above,
\[ 
\begin{split}
&\quad\|\nabla^n(\mathcal{J}_k^1+\mathcal{J}_k^2 + \mathcal{E}_k)\|_\infty\\
&\leq \|\nabla^n\mathcal{J}_k^1\|_\infty+\|\nabla^n\mathcal{J}_k^2\|_\infty+\|\nabla^n\mathcal{E}_k^1\|_\infty+\|\nabla^n\mathcal{E}_k^2\|_{L^\infty}+\|\nabla^n\mathcal{E}_k^3\|_\infty\\
&\lesssim_{\varepsilon,m} \sum_{j} N_{j,k+1}^{n -1+ \varepsilon} N_{J,k} \big(e^{-\frac14 M_{j,k+1}^2t} + M_{j,k+1}^{-m}N_{J,k}^{m} \big)\\
&\quad \, \,  +\sum_{j'<j} N^{n-1+\varepsilon}_{j,k+1}N_{j',k+1}e^{-\frac14N_{j,k+1}^2t}\\
&\quad \, \, +\sum_{j'<j} N_{j,k+1}^{-1}N_{j',k+1} N_{j,k+1}^{n-m}M_{j,k+1}^{m+\varepsilon}\\
&\quad \, \, +\sum_{j}\big( N_{j,k+1}^{n-1+\varepsilon}e^{-\frac14 N_{j,k+1}^2t}+N_{j,k+1}^{-2}M_{j,k+1}^{n+ 2+\varepsilon}\big)
+\sum_jN_{j,k+1}^{n+1+\varepsilon}\ell_{k+1}.
\end{split}
\]
Exactly as in \cites{Dai2026,CDP}, the hierarchy
\eqref{N_and_M_ordering} implies that after taking $A$, $b$, and $m$ large
enough, there exists $\alpha> 0$ such that for any $\varepsilon_0 >0$ and
$\bar n \in \mathbb{N}$,
\[
\|\nabla^n(\mathcal{J}_k^1+\mathcal{J}_k^2 + \mathcal{E}_k)\|_\infty \leq \varepsilon_0 N^{-\alpha}_{1,k+1}(t^{-\frac{n}{2} +\alpha}+1), \qquad n=1,2,\dots, \bar n,
\]
while for $n=0$ the same summation yields
\[
\|R_k(t)- \bar R_k(t)\|_{L^\infty} \lesssim_\varepsilon \varepsilon_0N_{1,k}^{-\alpha}+\mathbbm{1}_{t\leq t_k}N_{J, k}^{\varepsilon},
\]
as in Proposition~\ref{barv-I_estimate_proposition}.
\end{proof}

\begin{proposition}\label{difference_estimate_H}
For all sufficiently small $\alpha>0$, all $\varepsilon_0>0$, and every
$\bar n \in \mathbb{N}$, the parameters $A$ and $b$ may be chosen large enough
so that
\[
\|\nabla^n(H_k(t)- \bar H_k(t) )\|_{L^\infty} \leq_n \varepsilon_0 N_{1,k}^{-\alpha} (t^{-\frac{n}{2}+\alpha}+1), \qquad n=1,2,\dots, \bar n,
\]
\[
\|H_k(t)- \bar H_k(t)\|_{L^\infty} \lesssim_\varepsilon \varepsilon_0N_{1,k}^{-\alpha}+\mathbbm{1}_{t\leq t_k}N_{J, k}^{\varepsilon}.
\]
\end{proposition}
\begin{proof}
Recall the quantity $\mathcal I_{k,b}$ from
Proposition~\ref{barv-I_estimate_proposition}:
\[
\mathcal{I}_{k,b}
= -\int_0^{t} e^{(t-s)\Delta }\sum_{j\in \Lambda_b} N_{j,k+1}^2 e^{-2N_{j,k+1}^2s}
\mathcal R_1\div \!\left(a_{j,k+1,b}\eta^{j,\perp}\right) \, ds.
\]
By the triangle inequality,
\begin{equation}\label{eq:Hk_barHk_triangle}
\|\nabla^n(H_k-\bar H_k)\|_\infty
\le \|\nabla^n(H_k-\mathcal I_{k,b})\|_\infty+\|\nabla^n(\bar H_k-\mathcal I_{k,b})\|_\infty.
\end{equation}

We first estimate $H_k-\mathcal I_{k,b}$. The expansion is the same as in
Proposition~\ref{difference_estimate_proposition}, except that it is now
applied to the product $\bar v_{k+1}\bar h_{k+1}$. The unidirectional
contribution has the form
\[
\begin{split}
\bar v_{j,k+1}\bar h_{j,k+1}
&=N_{j,k+1}^2e^{-2N_{j,k+1}^2t}a_{j,k+1,b}\eta^{j,\perp}\\
&\quad+N_{j,k+1}^2e^{-2N_{j,k+1}^2t}a_{j,k+1,b}\big(\varphi_{j,k+1}^2\sin^2(N_{j,k+1}\eta^j\cdot x)-1\big)\eta^{j,\perp}\\
&\quad+\text{\emph{l.o.t.}},
\end{split}
\]
and the interactions of different directions are again lower-order terms.
Accordingly, we write
\[
H_k-\mathcal I_{k,b}=\widetilde{\mathcal J}_k^1+\widetilde{\mathcal J}_k^2+\widetilde{\mathcal E}_k,
\]
where
\[
\begin{split}
\widetilde{\mathcal J}_k^1
&=-\int_0^t e^{(t-s)\Delta}\sum_{j\in\Lambda_b}N_{j,k+1}^2e^{-2N_{j,k+1}^2s}\\
&\qquad\cdot\mathcal R_1\div\!\left(a_{j,k+1,b}(\varphi_{j,k+1}^2\sin^2(N_{j,k+1}\eta^j\!\cdot x)-1)\eta^{j,\perp}\right)\,ds,
\end{split}
\]
$\widetilde{\mathcal J}_k^2$ collects the different-direction interactions, and
$\widetilde{\mathcal E}_k$ collects all lower-order derivative and
mollifier-removal terms, including $\mathfrak{E}_{k,b}^3$ and analogous
contributions.

This is exactly the analogue of the decomposition in
Proposition~\ref{difference_estimate_proposition}, with the substitutions
\[
(\mathbb Q,\mathcal J_k^1,\mathcal J_k^2,\mathcal E_k,a_{j,k+1,u})
\rightsquigarrow
(\mathcal R_1\div,\widetilde{\mathcal J}_k^1,\widetilde{\mathcal J}_k^2,\widetilde{\mathcal E}_k,a_{j,k+1,b}).
\]
Hence the same oscillation, commutator, and mollification estimates from
\cites{Dai2026,CDP} together with
Proposition~\ref{difference_estimate_proposition} yield, for $n\ge1$,
\[
\|\nabla^n\widetilde{\mathcal J}_k^1\|_\infty
\lesssim \sum_j N_{J,k}^{1+\varepsilon}
\Big(N_{j,k+1}^{n-1}e^{-M_{j,k+1}^2t/4}+M_{j,k+1}^{-m+n}N_{J,k}^{m}\Big),
\]
and the corresponding $n=0$ bound follows exactly as in the estimate for
$\mathcal J_{k,u}^1$ above.

Likewise, $\widetilde{\mathcal J}_k^2$ satisfies the same bound as
$\mathcal J_k^2$, and $\widetilde{\mathcal E}_k$ satisfies the same bound as
$\mathcal E_k$, up to an absolute constant. Combining these estimates exactly
as at the end of Proposition~\ref{difference_estimate_proposition}, we obtain
the following: for any $\varepsilon_0>0$, $\bar n\in\mathbb N$, and
$\alpha\in(0,\frac1{10})$, after choosing $A,b,m$ sufficiently large,
\begin{equation}\label{eq:H_minus_I_bound}
\|\nabla^n(H_k-\mathcal I_{k,b})\|_\infty
\le \varepsilon_0 N_{1,k+1}^{-\alpha}\big(t^{-n/2+\alpha}+1\big),\qquad n=1,\dots,\bar n,
\end{equation}
and for $n=0$:
\[
\|H_k-\mathcal I_{k,b}\|_\infty
\lesssim_\varepsilon \varepsilon_0N_{1,k+1}^{-\alpha}+\mathbbm 1_{t\le t_k}N_{J,k}^{\varepsilon}.
\]

On the other hand, Proposition~\ref{barv-I_estimate_proposition} gives
\[
\|\nabla^n(\bar H_k-\mathcal I_{k,b})\|_\infty
\le \varepsilon_0 N_{1,k}^{-\alpha}\big(t^{-n/2+\alpha}+1\big),\qquad n=1,\dots,\bar n,
\]
and the corresponding $n=0$ bound. Substituting these estimates together with
\eqref{eq:H_minus_I_bound} into \eqref{eq:Hk_barHk_triangle}, and using that
$N_{1,k+1}^{-\alpha}\le N_{1,k}^{-\alpha}$, after adjusting constants we conclude
\[
\|\nabla^n(H_k-\bar H_k)\|_\infty
\le \varepsilon_0 N_{1,k}^{-\alpha}(t^{-n/2+\alpha}+1),\qquad n=1,\dots,\bar n,
\]
and
\[
\|H_k-\bar H_k\|_\infty
\lesssim_\varepsilon \varepsilon_0N_{1,k}^{-\alpha}+\mathbbm 1_{t\le t_k}N_{J,k}^{\varepsilon}.
\]
This proves the proposition.
\end{proof}

\subsection{Bounds on the solution and residual}

We now show that the principal part $(v,h)$ solves a forced system with small
forcing terms.

\begin{proposition}\label{prop_f_estimate}
For any $\varepsilon_0>0$, we may arrange that the principal part $(v,h)$ satisfies
    \begin{equation}\label{v_equation}
    \begin{split}
 \partial_tv-\Delta v+\mathbb P\div (v\otimes v)&=\mathbb P\div f_u,\\
 \partial_th-\Delta h+\mathbb P\div (vh)&=\mathbb P\div f_b,
 \end{split}
    \end{equation}
    for forcing terms $f_u$ and $f_b$ satisfying
\begin{align}\label{f_bound}
    \|\nabla^n f_u\|_{C^\kappa}+  \|\nabla^n f_b\|_{C^\kappa}\lesssim_n \varepsilon_0 (t^{-1 - \frac{n}{2} +\alpha}+1)
\end{align}
for some $0<\kappa<\alpha<1$.
In addition,
\begin{align}
    \|\nabla^nv_k(t)\|_{L^p}+ \|\nabla^nh_k(t)\|_{L^p}&\lesssim ((t/N_{1,k+1})^{\alpha}+2^{-k/p})t^{-\frac12(1+n)}+N_{1,k+1}^{-\alpha},\label{vk-pointwise-bounds}
    \end{align}
    and
    \begin{align}
    \|\nabla^nv(t)\|_{L^\infty}+\|\nabla^n\bar v(t)\|_{L^\infty}+ \|\nabla^nh(t)\|_{L^\infty}+\|\nabla^n\bar h(t)\|_{L^\infty}&\lesssim t^{-\frac12(1+n)}\label{v-pointwise-bounds}
\end{align}
for $n=0,1,2,\ldots,\overline n$, and
\begin{align}\label{v-critical-bounds}
    \|v\|_{L^1([t',t], t^{-\frac12}dt; L^\infty)}+\|v\|_{L^2([t',t]; L^\infty)}^2\lesssim 1+(\log A)^{-1}\log(t/{t'}),
\end{align}
\begin{align}\label{h-critical-bounds}
    \|h\|_{L^1([t',t], t^{-\frac12}dt; L^\infty)}+\|h\|_{L^2([t',t]; L^\infty)}^2\lesssim 1+(\log A)^{-1}\log(t/{t'}).
\end{align}
\end{proposition}

\begin{proof}
From the definitions in \eqref{app-prin} and the bound \eqref{eq:Dpsi_bounds},
we have
\begin{equation}\label{vk_bar_estimate}
\begin{split}
  &\quad  \|\nabla^n\overline v_k(t)\|_{p}+\|\nabla^n\overline h_k(t)\|_{p}\\
    &\lesssim |\Omega_k|^\frac1p\left(\|\nabla^n\overline v_k(t)\|_{\infty}+\|\nabla^n\overline h_k(t)\|_{\infty}\right)\\
    &\lesssim 2^{-k/p}\sum_jN_{j,k}^{1+n}e^{-N_{j,k}^2t}
    \end{split}
\end{equation}
which yields the estimate for $(\bar v, \bar h)$ in
\eqref{v-pointwise-bounds} after summing over $k\geq0$. Combining this with
Proposition~\ref{difference_estimate_proposition} and
Proposition~\ref{difference_estimate_H}, we get
\begin{equation}\label{alt_vk_bound}
\begin{split}
    \|\nabla^nv_k(t)\|_p&\leq \|\nabla^n(v_k(t)-\bar{v}_k(t))\|_p+\|\nabla^n\bar{v}_k(t)\|_p\\
    &\lesssim \|\nabla^{n+1}(R_k(t)-\bar{R}_k(t))\|_p+\|\nabla^n\bar{v}_k(t)\|_p\\
    &\lesssim N_{1,k+1}^{-\alpha}(t^{-\frac12-\frac n2+\alpha}+1)+2^{-k/p}\sum_jN_{j,k}^{1+n}e^{-N_{j,k}^2t},
\end{split}
\end{equation}
\begin{equation}\label{alt_hk_bound}
\begin{split}
    \|\nabla^nh_k(t)\|_p&\leq \|\nabla^n(h_k(t)-\bar{h}_k(t))\|_p+\|\nabla^n\bar{h}_k(t)\|_p\\
    &\lesssim \|\nabla^{n+1}(H_k(t)-\bar{H}_k(t))\|_p+\|\nabla^n\bar{h}_k(t)\|_p\\
    &\lesssim N_{1,k+1}^{-\alpha}(t^{-\frac12-\frac n2+\alpha}+1)+2^{-k/p}\sum_jN_{j,k}^{1+n}e^{-N_{j,k}^2t},
\end{split}
\end{equation}
and therefore \eqref{vk-pointwise-bounds}--\eqref{v-pointwise-bounds} follow.

To define $f_u$ and $f_b$, recall that for $k\geq0$,
\[
\begin{split}
v_k(t) &= -\int_0^{t} e^{(t-s)\Delta }\mathbb P\div (\bar v_{k+1}\otimes \bar v_{k+1})(s) \, ds\\
&=\int_0^{t} e^{(t-s)\Delta }\mathbb P\div (-v_{k+1}\otimes  v_{k+1}+f_{k+1,u})(s) \, ds,
\end{split}
\]
where
\[
f_{k,u}=v_{k} \otimes  v_{k}-\bar v_{k}\otimes \bar v_{k}
=(v_{k}-\bar v_{k}) \otimes v_{k} + \bar v_{k} \otimes (v_{k}-\bar v_{k}),
\]
and similarly
\[
\begin{split}
h_k(t) &= -\int_0^{t} e^{(t-s)\Delta }\div (\bar v_{k+1} \bar h_{k+1})(s) \, ds\\
&=\int_0^{t} e^{(t-s)\Delta }\div (-v_{k+1} h_{k+1}+f_{k+1,b})(s) \, ds,
\end{split}
\]
with
\[
f_{k,b}=v_{k} h_{k}-\bar v_{k} \bar h_{k}
=(v_{k}-\bar v_{k}) h_{k} + \bar v_{k} (h_{k}-\bar h_{k}).
\]
Summing over $k\geq0$, we obtain \eqref{v_equation} in mild form provided we set
\begin{equation}\label{def-fu-fb}
\begin{split}
f_u&=\sum_{k\geq0} f_{k,u}   +\sum_{k_1\neq k_2}v_{k_1}\otimes v_{k_2},\\
f_b&=\sum_{k\geq0} f_{k,b}   +\sum_{k_1\neq k_2}v_{k_1} h_{k_2}.
\end{split}
\end{equation}
Indeed, by \eqref{def_psi_0}, $\bar v_0$ is a shear flow, the coefficients
$a_{j,0,u}$ are constant, and $\bar h_0=0$. Thus the zero mode is absorbed into
$f_{0,u}$ and $f_{0,b}$ exactly as in \cites{Dai2026,CDP}. Furthermore,
Proposition~\ref{difference_estimate_proposition},
Proposition~\ref{difference_estimate_H}, and \eqref{v-pointwise-bounds} imply,
after shrinking $\alpha$ and fixing $0<\kappa<\alpha$,
\begin{align*}
\|\nabla^nf_{k,u}(t)\|_{C^\kappa} &\lesssim \sum_{i=0}^n\|\nabla^i(v_{k}-\bar v_{k})\|_{C^\kappa}(\|\nabla^{n-i}v_{k}\|_{C^\kappa}+\|\nabla^{n-i}\bar v_{k}\|_{C^\kappa})\\
&\lesssim\sum_{i=0}^n\varepsilon_0N_{1,k+1}^{-2\alpha}(t^{-\frac12(1+i+\kappa)+2\alpha}+1)(t^{-\frac12(1+n-i+\kappa)}+1)\\
&\lesssim \varepsilon_0N_{1,k+1}^{-2\alpha}(t^{-1-\frac n2+\alpha+(\alpha-\kappa)}+1),
\end{align*}
\begin{align*}
\|\nabla^nf_{k,b}(t)\|_{C^\kappa} &\lesssim \sum_{i=0}^n\|\nabla^i(v_{k}-\bar v_{k})\|_{C^\kappa}\|\nabla^{n-i}h_{k}\|_{C^\kappa}+\|\nabla^i(h_{k}-\bar h_{k})\|_{C^\kappa}\|\nabla^{n-i}\bar v_{k}\|_{C^\kappa}\\
&\lesssim\sum_{i=0}^n\varepsilon_0N_{1,k+1}^{-2\alpha}(t^{-\frac12(1+i+\kappa)+2\alpha}+1)(t^{-\frac12(1+n-i+\kappa)}+1)\\
&\lesssim \varepsilon_0N_{1,k+1}^{-2\alpha}(t^{-1-\frac n2+\alpha+(\alpha-\kappa)}+1).
\end{align*}
Summing over $k$, we obtain the bound claimed in \eqref{f_bound} for the first
sum in $(f_u,f_b)$.

For the mixed-frequency terms in \eqref{def-fu-fb}, we use the same four-way
decomposition as in \cites{Dai2026,CDP}, based on the two summands in
\eqref{alt_vk_bound} and \eqref{alt_hk_bound}. Since $v_k$ and $h_k$ obey the
same pointwise bounds, the scalar interaction $v_{k_1}h_{k_2}$ is handled in
exactly the same way as $v_{k_1}\otimes v_{k_2}$. The low-low and mixed pieces
contribute
\[
t^{-1-\frac n2+2\alpha-\kappa},
\]
just as in those references. For the high-high contribution, using
\eqref{N_definition} and
\[
N_{J,k_2-1}\sim N_{j,k_2}^{b^{-j/J}},
\]
we obtain
\begin{align*}
\sum_{i=0}^n\sum_{j}\sum_{k_2}N_{J,k_2-1}^{1+i+\kappa}N_{j,k_2}^{1+n-i+\kappa}e^{-N_{j,k_2}^2t}
&\lesssim \sum_j\sum_{k_2}N_{j,k_2}^{n+2-(1-b^{-j/J})+2\kappa}e^{-N_{j,k_2}^2t}\\
&\lesssim t^{-1-\frac n2+\frac12(1-b^{-1/J}-4\kappa)}.
\end{align*}
Consequently,
\begin{align*}
\Big\|\nabla^n\sum_{k_1\neq k_2}v_{k_1}\otimes v_{k_2}\Big\|_{C^\kappa}
+\Big\|\nabla^n\sum_{k_1\neq k_2}v_{k_1} h_{k_2}\Big\|_{C^\kappa}
\lesssim t^{-1-\frac n2+2\alpha-\kappa}+t^{-1-\frac n2+\frac12(1-b^{-1/J}-4\kappa)}.
\end{align*}
Taking $\alpha,\kappa>0$ smaller if needed, we conclude \eqref{f_bound}.

To prove \eqref{v-critical-bounds}, it suffices to control the
$L^1([t',t],t^{-1/2}dt;L^\infty)$ norm and then interpolate with
\eqref{v-pointwise-bounds}. Starting from \eqref{alt_vk_bound}, we split the
resulting integral exactly as in \cites{Dai2026,CDP} into the regions
$N_{j,k}<t^{-1/2}$, $t^{-1/2}\le N_{j,k}\le (t')^{-1/2}$, and
$N_{j,k}>(t')^{-1/2}$. The low- and high-frequency pieces are $O(1)$, while
the middle range is controlled by the number of indices $k$ satisfying
\[
t^{-1/2}\le N_{j,k}\le (t')^{-1/2}.
\]
Since \eqref{N_definition} implies
\[
\frac{N_{j,k+1}}{N_{j,k}}\sim A^{b^{k+1+(j-1)/J}-b^{k+(j-1)/J}}\gtrsim A^{b-1},
\]
and that counting argument yields
\[
\|v\|_{L^1([t',t],t^{-\frac12}dt;L^\infty)}
\lesssim 1+\log_{A^{b-1}}\!\big((t')^{-1/2}/t^{-1/2}\big)
\lesssim 1+(\log A)^{-1}\log(t/t').
\]
This proves \eqref{v-critical-bounds}. The argument for
\eqref{h-critical-bounds} is identical, with \eqref{alt_hk_bound} replacing
\eqref{alt_vk_bound}. The corresponding $L^2_tL^\infty_x$ bounds then follow
by interpolation with \eqref{v-pointwise-bounds}, completing the proof.
\end{proof}

\section{Construction of the corrector}\label{sec:corrector}

Let $(U,H)$ be any classical solution to system \eqref{eq-main} on the time interval $[0,T]$ which, once again, we replace with $[-T_*, T-T_*]$ by translation. We may define
\begin{align*}
    \sup_{0\leq n\leq 10}\left(\|\nabla^nU\|_{L_{t,x}^\infty(\mathbb T^2\times[0,T-T_*])}+\|\nabla^n H\|_{L_{t,x}^\infty(\mathbb T^2\times[0,T-T_*])}\right)\eqcolon C_{U,H}<\infty.
\end{align*}
The family of solutions $(u^{(\sigma)},b^{(\sigma)})$ will arise by modulating a frequency scale $N_0>0$ (see Section~\ref{sec:proof}). Having constructed $(v,h)$ on $\mathbb T^2\times[0,\infty)$, we extend them by $0$ to the full time interval $\mathbb R$, then rescale them to $(v^{N_0},h^{N_0})$, which we regard as $2\pi/N_0$-periodic fields on $\mathbb R^2\times[0,\infty)$. Here, for an arbitrary pair $(V_u, V_b)$ on $\mathbb R^2\times \mathbb R^2$, we define the rescaling
\begin{align*}
    V_u^{N_0}(x,t)\coloneqq N_0V_u(N_0x,N_0^2t), \qquad V_b^{N_0}(x,t)\coloneqq N_0V_b(N_0x,N_0^2t).
\end{align*}
Recall that critical norms such as $L^\infty((0,\infty); \dot W^{-1,\infty}(\mathbb R^2))$ are invariant under this transformation.

From there, we construct the exact solution of \eqref{eq-main} as
\[
u = U + v^{N_0} + w^{N_0},\qquad b = H + h^{N_0} + \zeta^{N_0},
\]
with $(w,\zeta)$ a small corrector pair. To simplify the computation, we instead construct the rescaled solution
\[
u^{1/N_0}=U^{1/N_0}+v+w,\qquad b^{1/N_0}=H^{1/N_0}+h+\zeta.
\]
Then $(w,\zeta)$ should satisfy
\[
\begin{split}
&\quad\partial_tw-\Delta w+\mathbb P\div\!\Big(w\otimes w+2(U^{1/N_0}+v)\odot w\Big)\\
&=-\mathbb P\div\!\Big(f_u+2U^{1/N_0}\odot v\Big),\\
&\quad\partial_t\zeta-\Delta \zeta+\div\!\Big(w\zeta+(U^{1/N_0}+v)\zeta +w(H^{1/N_0}+h)\Big)\\
&=-\div\!\Big(f_b+U^{1/N_0} h+vH^{1/N_0}\Big),\\
(w,\zeta)|_{t=0}&=(0,0).
\end{split}
\]
After the translation and rescaling, the lifetime of $(U^{1/N_0},H^{1/N_0})$ is $[-N_0^2T_*,N_0^2(T-T_*)]$.

\subsection{Semigroup estimates}

We next record the semigroup estimates needed for the fixed point construction
of the coupled corrector pair $(w,\zeta)$. Let
\[
\bar T=\min\{N_0^2(T-T_*),\bar C\}
\]
for some large $\bar C>1$ to be chosen later. For suitable $\alpha$ and
$\kappa$, define
\begin{equation}\label{def-X}
\begin{split}
X&=\Big\{V\in C^0((0,\bar T]; C^{1,\kappa}(\mathbb T^2; \mathbb R^2)): \\
& \qquad\qquad \|V\|_{X}:= \sup_{t\in(0,\bar T]}(t^{\frac{1-\alpha}2} \|V\|_{L^\infty}+t^{\frac{2-\alpha}2} \|\nabla V\|_{C^{\kappa}})<\infty\Big\}.
\end{split}
\end{equation}
and
\begin{equation}\label{def-Y}
\begin{split}
Y&=\Big\{\phi\in C^0((0,\bar T]; C^{1,\kappa}(\mathbb T^2; \mathbb R^{2\times d})): \\
& \qquad\qquad \|\phi\|_{Y}:= \sup_{t\in(0,\bar T]}(t^{1-\alpha} \|\phi\|_{L^\infty}+t^{\frac32-\alpha} \|\nabla \phi\|_{C^{\kappa}})<\infty\Big\}
\end{split}
\end{equation}
where $d=1$ or 2.
Define the product spaces
\[
\mathcal X\coloneqq X\times X,\qquad \mathcal Y\coloneqq Y\times Y, \]
with
\[
\|(W,Z)\|_{\mathcal X}\coloneqq \|W\|_X+\|Z\|_X,\quad \|(\phi_u,\phi_b)\|_{\mathcal Y}\coloneqq \|\phi_u\|_Y+\|\phi_b\|_Y.
\]
This Banach-space setup and the semigroup estimate below follow the same
corrector construction as in \cites{Dai2026,CDP}; we retain the notation here
because it will also be used in the fixed point argument. The following
product estimate is immediate from
\[
\|gh\|_{C^{1,\kappa}}\lesssim \|g\|_{C^{1,\kappa}}\|h\|_{L^\infty}+\|g\|_{L^\infty}\|h\|_{C^{1,\kappa}}.
\]
\begin{lemma}\label{product_X_Y_lemma}
    If $g,h\in X$, then $g\otimes h,g\odot h\in Y$ with
    \begin{align*}
        \|g\otimes h\|_Y+\|g\odot h\|_Y\lesssim \|g\|_X\|h\|_X.
    \end{align*}
\end{lemma}

Set
\[
\tilde v\coloneqq U^{1/N_0}+v,\qquad \tilde h\coloneqq H^{1/N_0}+h.
\]
For $\Phi=(\phi_u,\phi_b)\in\mathcal Y$ and $0<t'\leq t\leq \bar T$, let
\[
\mathbb S(t,t')\Phi=(W,Z)
\]
be the solution of
\begin{equation}\label{semi-group}
\begin{split}
\partial_tW-\Delta W+\mathbb P\div\!\Big(2\tilde v\odot W\Big)&=0,\\
\partial_tZ-\Delta Z+\div\!\Big(\tilde v Z+W \tilde h\Big)&=0,\\
(W,Z)\vert_{t=t'}=\big(\mathbb P\div \phi_u(t'),\,  \div \phi_b(t')\big)&.
\end{split}
\end{equation}
As before, we suppress the dependence of $\mathbb S(t,t')$ on the parameter
$1/N_0$.

\begin{proposition}\label{prop-semi}
For any $\Phi=(\phi_u,\phi_b)\in\mathcal Y$ and $0<t'\leq t\leq \bar T$, we have
\begin{equation*}
\begin{aligned}
\|W(t)\|_{L^\infty}+\|Z(t)\|_{L^\infty}
&+(t-t')^{\frac12}\Big(\|\nabla W(t)\|_{C^\kappa}+\|\nabla Z(t)\|_{C^\kappa}\Big)\\
&\lesssim_{\bar C} t^{-\frac12}(t')^{-1+\alpha}(t/t')^{\epsilon}\|\Phi\|_{\mathcal Y}.
\end{aligned}
\end{equation*}
\end{proposition}

\begin{proof}
The proof is a direct adaptation of the corresponding semigroup estimates in
\cites{Dai2026,CDP}. Define
\[
\mathcal H(t)\coloneqq t^{1/2}\big(\|W(t)\|_{L^\infty}+\|Z(t)\|_{L^\infty}\big).
\]
Applying Duhamel to \eqref{semi-group}, splitting $[t',t]$ exactly as in those
references, and using the heat kernel bound, we obtain
\[
\mathcal H(t)\lesssim (t')^{-1+\alpha}\|\Phi\|_{\mathcal Y}
+\int_{t'}^t\Big(s^{-1/2}+(t-s)^{-1/2}\Big)
\Big(C_{U,H}+\|v(s)\|_{L^\infty}+\|h(s)\|_{L^\infty}\Big)\mathcal H(s)\,ds,
\]
The only extra contribution relative to \cite{CDP} is the pair $(Z,\tilde h)$,
which is handled exactly as the second component in \cite{Dai2026}, since
\[
\|\tilde v(s)\|_{L^\infty}+\|\tilde h(s)\|_{L^\infty}\lesssim C_{U,H}+\|v(s)\|_{L^\infty}+\|h(s)\|_{L^\infty},
\]
Now \eqref{v-critical-bounds} and \eqref{h-critical-bounds} provide the same
logarithmic control used there, and fractional Gr\"onwall yields
\[
\|W(t)\|_{L^\infty}+\|Z(t)\|_{L^\infty}\lesssim t^{-1/2}(t')^{-1+\alpha}(t/t')^{\epsilon/2}\|\Phi\|_{\mathcal Y}.
\]
Finally, differentiating \eqref{semi-group} and repeating the argument with
\begin{equation}\label{heat}
\|e^{t\Delta}\nabla^m \mathbb P g\|_{C^{r}}\lesssim t^{-\frac{m+r-s}2}\| g\|_{C^{s}},
\end{equation}
gives
\[
(t-t')^{1/2}\big(\|\nabla W(t)\|_{C^\kappa}+\|\nabla Z(t)\|_{C^\kappa}\big)\lesssim t^{-1/2}(t')^{-1+\alpha}(t/t')^{\epsilon}\|\Phi\|_{\mathcal Y},
\]
after enlarging the implicit constant and replacing $\epsilon/2$ by $\epsilon$.
\end{proof}


\subsection{Fixed point argument}

We now carry out the fixed point argument that constructs the corrector. As in
\cites{Dai2026,CDP}, the key input is the combination of the semigroup estimate
from Proposition~\ref{prop-semi} with the quadratic structure of the corrector
map.
\begin{proposition}\label{w-exists-fixed-point-proposition}
There exists $\varepsilon>0$ such that, for every $\delta\in(0,\varepsilon)$
and all sufficiently large $A$, one can find $w,\zeta\in B_X(0,\delta)$ with
\[
u^{1/N_0}=U^{1/N_0}+v+w,\qquad b^{1/N_0}=H^{1/N_0}+h+\zeta
\]
solving \eqref{eq-main} on $[0,\bar T]\times\mathbb R^2$. Moreover, if $\bar C$
and $N_0$ are chosen sufficiently large depending on $(U,H)$ and $T$, then the
solution $(u^{1/N_0},b^{1/N_0})$ extends to a classical solution on $[0,T]$.
\end{proposition}

\begin{proof}
This is the same contraction argument as in \cites{Dai2026,CDP}, so we record
only the terms specific to the passive-scalar setting. For $(W,Z)\in\mathcal X$,
define
\begin{align*}
\Phi_u(W,Z)&\coloneqq W\otimes W+f_u+2U^{1/N_0}\odot v,\\
\Phi_b(W,Z)&\coloneqq W Z+f_b+U^{1/N_0} h+vH^{1/N_0},
\end{align*}
and set $\Psi(W,Z)\coloneqq (\Phi_u(W,Z),\Phi_b(W,Z))\in\mathcal Y$. We then define
\[
\mathcal F(W,Z)(t)\coloneqq-\int_0^t\mathbb S(t,t')\Psi(W,Z)\,dt'.
\]

Fix $\epsilon\in(0,\alpha)$. Proposition~\ref{prop-semi}, together with the
same time integration as in \cites{Dai2026,CDP}, gives
\[
\|\mathcal F(W,Z)\|_{\mathcal X}\lesssim_{\bar C}\|\Psi(W,Z)\|_{\mathcal Y}.
\]

By Lemma~\ref{product_X_Y_lemma}, \eqref{f_bound}, smoothness of $(U,H)$, and \eqref{v-pointwise-bounds},
\[
\|W\otimes W\|_Y+\|W Z\|_Y\lesssim \|(W,Z)\|_{\mathcal X}^2,
\qquad
\|f_u\|_Y+\|f_b\|_Y\lesssim \varepsilon_0,
\]
\[
\|U^{1/N_0}\|_{L_t^\infty C_x^{1,\kappa}}+\|H^{1/N_0}\|_{L_t^\infty C_x^{1,\kappa}}\lesssim N_0^{-1}C_{U,H},
\]
\[
\|U^{1/N_0}\odot v\|_Y+\|U^{1/N_0} h\|_Y+\|v H^{1/N_0}\|_Y
\lesssim_{\bar C,U,H} N_0^{-1}.
\]
Hence the only terms not already present in the companion arguments are handled
by the same product bound together with the same smallness of the background
interactions. Therefore
\[
\|\Psi(W,Z)\|_{\mathcal Y}\le C_{\bar C,U,H}\Big(\|(W,Z)\|_{\mathcal X}^2+\varepsilon_0+N_0^{-1}\Big),
\]
and consequently
\[
\|\mathcal F(W,Z)\|_{\mathcal X}\le C_{\bar C,U,H}\Big(\|(W,Z)\|_{\mathcal X}^2+\varepsilon_0+N_0^{-1}\Big).
\]

For $(W_i,Z_i)\in B_{\mathcal X}(0,\delta)$, $i=1,2$, expanding the quadratic
differences exactly as in \cites{Dai2026,CDP} yields
\[
\|\Psi(W_1,Z_1)-\Psi(W_2,Z_2)\|_{\mathcal Y}
\lesssim \delta\,\|(W_1-W_2,Z_1-Z_2)\|_{\mathcal X}.
\]
Applying the linear estimate once more,
\[
\|\mathcal F(W_1,Z_1)-\mathcal F(W_2,Z_2)\|_{\mathcal X}
\le C_{\bar C,U,H}\delta\,\|(W_1-W_2,Z_1-Z_2)\|_{\mathcal X}.
\]

Choose $\delta>0$ such that $C_{\bar C,U,H}\delta\le \frac12$. Then choose $A\gg1$ (hence $\varepsilon_0$ small through \eqref{f_bound}) and $N_0$ large so that
\[
C_{\bar C,U,H}\big(\delta^2+\varepsilon_0+N_0^{-1}\big)\le \delta.
\]
It follows that $\mathcal F$ is a contraction on $B_{\mathcal X}(0,\delta)$ and
therefore has a unique fixed point $(w,\zeta)\in B_{\mathcal X}(0,\delta)$.

By construction, $(w,\zeta)$ solves the coupled corrector system, and hence
$(u^{1/N_0},b^{1/N_0})$ solves \eqref{eq-main} on $[0,\bar T]$. Finally, choose
$\bar C$ and then $N_0$ large enough so that
\[
\bar T=\min\{N_0^2(T-T_*),\bar C\}=N_0^2(T-T_*).
\]
Rescaling back yields a classical solution on $[0,T]$.
\end{proof}

\section{Proof of Theorem \ref{main-thm}}\label{sec:proof}

We now verify that the solution produced in Section \ref{sec:corrector} satisfies each conclusion in Theorem \ref{main-thm}.

We begin with a general gluing statement ensuring that piecewise classical solutions remain weak solutions provided the two traces agree distributionally at the joining time.

\begin{lemma}\label{le-weak}
Let $u,b\in L^2(\mathbb{T}^2 \times [0,T])\cap L^\infty([0,T];H^s(\mathbb{T}^2))$ for some $s\in \mathbb{R}$ be such that $(u,b)|_{[0,T_*)}$ and $(u,b)|_{(T_*,T]}$ are classical solutions of \eqref{eq-main} on $[0,T_*)$ and $(T_*,T]$ respectively. If
\[
\lim_{t\to T_*^-} (u(t),b(t)) = \lim_{t\to T_*^+} (u(t),b(t)), \qquad \text{in} \quad \mathcal{D}'(\mathbb{T}^2),
\]
then $u$ is a weak solution of the Navier--Stokes equations on $[0,T]$.
\end{lemma}

Next we check the borderline space bounds claimed for the constructed pair.

\begin{proposition}\label{prop:borderline}
    With $(u,b)$ as defined in Section \ref{sec:corrector}, we have
    \[
    u,b\in L_t^{2,\infty}L_x^\infty\cap L_{t}^2L_x^p
    \]
    for all $p<\infty$. Furthermore, both $u$ and $b$ lie in the Koch--Tataru path space $X_{T-T_*}$ and are weak-* continuous in time into $BMO^{-1}$. 
 \end{proposition}

\begin{proof}
Recall that the solution constructed in Section~\ref{sec:corrector} has the form
\[
u=U+v^{N_0}+w^{N_0},\qquad b=H+h^{N_0}+\zeta^{N_0}.
\]
After rescaling, it is enough to work with
\[
u^{1/N_0}=U^{1/N_0}+v+w,\qquad b^{1/N_0}=H^{1/N_0}+h+\zeta.
\]

\textbf{Step 1:} $\mathbf{X_{T-T_*}}$ (hence weak-* continuity into $\mathbf{BMO^{-1}}$) for both $u,b$.
We use the implication
\begin{align}\label{X_T_criterion}
\sup_{t\in(0,T-T_*]}t^{\frac12-\epsilon}\|f(t)\|_{L^\infty}<\infty\Longrightarrow f\in X_{T-T_*}.
\end{align}
Write
\begin{align*}
u^{1/N_0}&=U^{1/N_0}+w+u_1+u_2+u_3,\\
b^{1/N_0}&=H^{1/N_0}+\zeta+b_1+b_2+b_3,
\end{align*}
with
\begin{align*}
u_1&=\sum_{j,k}\big(e^{t\Delta}N_{j,k}\Delta\psi_{j,k,u}+\bar v_{j,k}\big),&
u_2&=v-\bar v,&
u_3&=-e^{t\Delta}\sum_{j,k}N_{j,k}\Delta\psi_{j,k,u},\\
b_1&=\sum_{j,k}\big(e^{t\Delta}N_{j,k}\Delta\psi_{j,k,b}+\bar h_{j,k}\big),&
b_2&=h-\bar h,&
b_3&=-e^{t\Delta}\sum_{j,k}N_{j,k}\Delta\psi_{j,k,b}.
\end{align*}
Since $U$ and $H$ are smooth, they belong to $X_{T-T_*}$, and Proposition~\ref{w-exists-fixed-point-proposition} yields the same conclusion for $w$ and $\zeta$. For $(u_1,b_1)$, Lemma~\ref{l:commutator} gives
\[
\|u_1(t)\|_\infty+\|b_1(t)\|_\infty\lesssim t^{-1/2+\epsilon}+1.
\]
The same $t^{-1/2+\epsilon}$ bound for $(u_2,b_2)$ follows from Proposition~\ref{difference_estimate_proposition} and Proposition~\ref{difference_estimate_H}. For $(u_3,b_3)$, the usual BMO argument based on \eqref{eq:Dpsi_bounds}, \eqref{Omega_volume_estimate}, and \eqref{N_and_M_ordering} shows that $(u_3,b_3)\in X_{T-T_*}$. Consequently,
\[
u,b\in X_{T-T_*}.
\]
Each summand is therefore weak-* continuous in time with values in $BMO^{-1}$.

\textbf{Step 3:} $\mathbf{L_t^{2,\infty}L_x^\infty\cap L_t^2L_x^p}$ for both $u,b$.
The $X$-bounds imply
\[
\|U^{1/N_0}(t)\|_\infty+\|H^{1/N_0}(t)\|_\infty+\|w(t)\|_\infty+\|\zeta(t)\|_\infty\lesssim t^{-1/2+\epsilon}.
\]
On the other hand, \eqref{v-pointwise-bounds} with $n=0$ yields
\[
v,h\in L_t^{2,\infty}L_x^\infty.
\]
Hence $u^{1/N_0}$ and $b^{1/N_0}$ lie in $L_t^{2,\infty}L_x^\infty$. In addition, \eqref{vk-pointwise-bounds} gives, for each $k$ and every $p<\infty$,
\[
\|v_k\|_{L_t^2L_x^p}+\|h_k\|_{L_t^2L_x^p}\lesssim N_{1,k+1}^{-\alpha}+2^{-k/p}.
\]
Summing over $k$ shows that $v,h\in L_t^2L_x^p$, and therefore
\[
u,b\in L_t^{2,\infty}L_x^\infty\cap L_t^2L_x^p,\qquad \forall p<\infty.
\]
This proves the proposition.
\end{proof}

\textbf{Proof of Theorem \ref{main-thm}:}
Section~\ref{sec:corrector} produces a solution of the form
\[
u=U+v^{N_0}+w^{N_0},\qquad b=H+h^{N_0}+\zeta^{N_0},
\]
which solves \eqref{eq-main} classically away from the gluing time. The distributional matching argument from Lemma~\ref{le-weak} then shows that $(u,B)$ is a weak solution of \eqref{eq-main} with initial data
\[
(u_0,b_0)=\big(U(0,\cdot),H(0,\cdot)\big).
\]

For the upper Type-I estimate, smoothness of $(U,H)$ together with \eqref{v-pointwise-bounds} gives
\[
\|U(t)\|_{L^\infty}+\|H(t)\|_{L^\infty}+\|v^{N_0}(t)\|_{L^\infty}+\|h^{N_0}(t)\|_{L^\infty}\lesssim t^{-1/2}.
\]
Moreover, the fixed-point bound in $X$ yields
\[
\|w^{N_0}(t)\|_{L^\infty}+\|\zeta^{N_0}(t)\|_{L^\infty}\lesssim t^{-1/2+\alpha/2}\lesssim t^{-1/2},
\]
so
\[
\|u(t)\|_{L^\infty}+\|b(t)\|_{L^\infty}\lesssim t^{-1/2},\qquad 0<t<T.
\]

For the lower bound, the principal velocity profile furnishes a sequence $t_n\to0^+$ such that
\[
\|v(t_n)\|_{L^\infty}\sim e^{-1}t_n^{-1/2}.
\]
Therefore,
\[
\|u(t_n)\|_{L^\infty}
\ge \|v^{N_0}(t_n)\|_{L^\infty}-\|U(t_n)\|_{L^\infty}-\|w^{N_0}(t_n)\|_{L^\infty}
\gtrsim t_n^{-1/2}.
\]
This gives the claimed instantaneous blowup. The borderline-space assertions for both components follow from Proposition~\ref{prop:borderline}.

\section{A perturbative route for genuine 2D MHD}
\label{sec:perturbative-2dmhd}

In this section, we prove Theorem \ref{thm-2dmhd} for the 2D MHD using a perturbative approach based on the 2D Navier--Stokes blowup profile established in the work \cite{CDP} of the first author and collaborators. The next two propositions and lemma provide a perturbative route. 

\begin{proposition}[Induction semigroup estimate around a 2D NSE background]
\label{prop:semi-2dmhd-sketch}
Fix $\bar T>0$, and let $X$, $Y$, and $\mathcal Y$ be the spaces introduced
later in Section~\ref{sec:corrector}, with parameters $\alpha,\kappa$ as in
\eqref{def-X}--\eqref{def-Y}. Let $u^{NS}$ be a divergence-free vector field on
$[0,\bar T]\times\mathbb T^2$ such that
\[
\|u^{NS}\|_X\leq C_{NS},
\]
and assume moreover that, for every $0<t'\leq t\leq \bar T$,
\[
\|u^{NS}\|_{L^1([t',t],s^{-1/2}ds;L^\infty(\mathbb T^2))}
\leq C_{NS}\big(1+\log(t/t')\big).
\]
For $0<t'\leq t\leq \bar T$ and spatial tensor fields
$f_u,f_B:\mathbb T^2\to\mathbb R^{2\times 2}$, let
\[
\mathbb M_{u^{NS}}(t,t')(f_u,f_B)=(W,G)
\]
solve
\begin{equation}\label{eq:semi-2dmhd-sketch}
\begin{split}
\partial_t W-\Delta W+\mathbb P\div(2u^{NS}\odot W)&=0,\\
\partial_t G-\Delta G+\div(u^{NS}\otimes G-G\otimes u^{NS})&=0,\\
(W,G)|_{t=t'}&=(\mathbb P\div f_u,\, \div f_B).
\end{split}
\end{equation}
Then there exists $\epsilon=\epsilon(C_{NS},\alpha,\kappa)>0$ such that, for
all $\Phi=(\phi_u,\phi_B)\in\mathcal Y$, the pair
\[
(W,G)=\mathbb M_{u^{NS}}(t,t')(\phi_u(t'),\phi_B(t'))
\]
satisfies
\[
\begin{aligned}
\|W(t)\|_{L^\infty}+\|G(t)\|_{L^\infty}
&+(t-t')^{\frac12}\Big(\|\nabla W(t)\|_{C^\kappa}+\|\nabla G(t)\|_{C^\kappa}\Big)\\
&\lesssim_{C_{NS},\bar T} t^{-1/2}(t')^{-1+\alpha}(t/t')^\epsilon
\|\Phi\|_{\mathcal Y}.
\end{aligned}
\]
\end{proposition}

\begin{proof}
For fixed $0<t'\leq t\leq \bar T$ and $\Phi=(\phi_u,\phi_B)\in\mathcal Y$, set
\[
\Phi_{t'}\coloneqq (\phi_u(t'),\phi_B(t')),
\]
so that
\[
(W,G)=\mathbb M_{u^{NS}}(t,t')\Phi_{t'}
\]
solves \eqref{eq:semi-2dmhd-sketch} with initial data
\[
(\mathbb P\div\phi_u(t'),\,\div\phi_B(t')).
\]
Define
\[
\mathcal H(t)\coloneqq t^{1/2}\big(\|W(t)\|_{L^\infty}+\|G(t)\|_{L^\infty}\big).
\]
Applying Duhamel to \eqref{eq:semi-2dmhd-sketch} and using the heat-kernel
bound exactly as in the proof of Proposition~\ref{prop-semi}, we obtain
\[
\begin{split}
\mathcal H(t)
&\lesssim (t')^{-1+\alpha}\|\Phi\|_{\mathcal Y}\\
&\quad+\int_{t'}^t\Big(s^{-1/2}+(t-s)^{-1/2}\Big)
\|u^{NS}(s)\|_{L^\infty}\mathcal H(s)\,ds.
\end{split}
\]
The $X$-bound implies the pointwise estimate
\[
\|u^{NS}(s)\|_{L^\infty}\lesssim C_{NS}s^{-\frac{1-\alpha}{2}},
\]
while the additional weighted $L^1_tL^\infty_x$ hypothesis provides exactly
the logarithmic control used in the proof of Proposition~\ref{prop-semi}.
Therefore the same fractional Gr\"onwall argument as in
Proposition~\ref{prop-semi} yields
\[
\|W(t)\|_{L^\infty}+\|G(t)\|_{L^\infty}
\lesssim_{C_{NS},\bar T}
t^{-1/2}(t')^{-1+\alpha}(t/t')^{\epsilon/2}\|\Phi\|_{\mathcal Y}
\]
for some $\epsilon=\epsilon(C_{NS})>0$.

To estimate derivatives, differentiate the Duhamel formula and use the heat
estimate
\[
\|e^{(t-s)\Delta}\nabla^m \mathbb P g\|_{C^r}
\lesssim (t-s)^{-\frac{m+r-\sigma}{2}}\|g\|_{C^\sigma}.
\]
Since the coefficients in \eqref{eq:semi-2dmhd-sketch} are linear combinations
of terms of the form $u^{NS}\otimes(\cdot)$ and $(\cdot)\otimes u^{NS}$, the
$X$-bound on $u^{NS}$ supplies the required Hölder control in the product
estimate, and the same argument gives
\[
(t-t')^{\frac12}\big(\|\nabla W(t)\|_{C^\kappa}+\|\nabla G(t)\|_{C^\kappa}\big)
\lesssim_{C_{NS},\bar T}
t^{-1/2}(t')^{-1+\alpha}(t/t')^\epsilon\|\Phi\|_{\mathcal Y},
\]
after enlarging the implicit constant and replacing $\epsilon/2$ by
$\epsilon$.
\end{proof}

For the magnetic initial data, define
\[
X_0\coloneqq \Big\{H_0\in \mathcal D'(\mathbb T^2;\mathbb R^2): \div H_0=0,\
\|H_0\|_{X_0}<\infty\Big\},
\]
where
\begin{equation}\label{def-X0-2dmhd}
\|H_0\|_{X_0}\coloneqq \sup_{0<t\leq \bar T}\Big(
t^{\frac{1-\alpha}2}\|e^{t\Delta}H_0\|_{L^\infty}
+t^{\frac{2-\alpha}2}\|\nabla e^{t\Delta}H_0\|_{C^\kappa}
\Big).
\end{equation}
\begin{remark}[A concrete model class for $X_0$]
The norm \eqref{def-X0-2dmhd} is a heat-semigroup realization of the negative
H\"older/Besov scale $B_{\infty,\infty}^{\alpha-1}(\mathbb T^2)$. In
particular, the standard heat-flow characterization gives
\[
\|H_0\|_{X_0}\lesssim \|H_0\|_{B_{\infty,\infty}^{\alpha-1}}
\]
for every divergence-free $H_0\in B_{\infty,\infty}^{\alpha-1}(\mathbb T^2;
\mathbb R^2)$. A particularly concrete subclass is obtained by taking
\[
H_0=\nabla^\perp \psi_0,\qquad \psi_0\in C^\alpha(\mathbb T^2),
\]
for which
\[
\|H_0\|_{X_0}\lesssim \|\psi_0\|_{C^\alpha}.
\]
Thus the smallness assumption in
Proposition~\ref{prop:2d-mhd-perturbative-sketch} may be read as a small
$C^\alpha$ condition on the initial magnetic stream function.
\end{remark}


\begin{proposition}[Linear induction estimate from the magnetic initial data]
\label{prop:Hlin-from-data-sketch}
Let $u^{NS}$ satisfy the hypotheses of Proposition~\ref{prop:semi-2dmhd-sketch},
and let $H_0\in X_0$. Then the system
\begin{equation}\label{eq:Hlin-linear}
\partial_t H^{lin}-\Delta H^{lin}
+\div(u^{NS}\otimes H^{lin}-H^{lin}\otimes u^{NS})=0,
\qquad H^{lin}|_{t=0}=H_0,
\end{equation}
has a unique solution $H^{lin}\in X$. Moreover,
\[
\|H^{lin}\|_X\leq C_0\|H_0\|_{X_0}
\]
for some constant $C_0=C_0(C_{NS},\alpha,\kappa,\bar T)>0$.
\end{proposition}

\begin{proof}
By Duhamel,
\[
H^{lin}(t)=e^{t\Delta}H_0-\int_0^t e^{(t-s)\Delta}\div\big(
u^{NS}\otimes H^{lin}-H^{lin}\otimes u^{NS}
\big)(s)\,ds.
\]
Set
\[
\mathcal K(t)\coloneqq
t^{\frac{1-\alpha}2}\|H^{lin}(t)\|_{L^\infty}
+t^{\frac{2-\alpha}2}\|\nabla H^{lin}(t)\|_{C^\kappa}.
\]
The free heat term is bounded by $\|H_0\|_{X_0}$ by definition. Applying the
heat estimate \eqref{heat} to the Duhamel term, splitting the time integral as
in the proof of Proposition~\ref{prop-semi}, and using the $X$-bound on
$u^{NS}$ for the Hölder-product estimate, we obtain
\[
\mathcal K(t)\lesssim \|H_0\|_{X_0}
+\int_0^t\Big(s^{-1/2}+(t-s)^{-1/2}\Big)
\|u^{NS}(s)\|_{L^\infty}\mathcal K(s)\,ds.
\]
The additional weighted $L^1_tL^\infty_x$ hypothesis on $u^{NS}$ therefore
gives the same logarithmic control as in Proposition~\ref{prop:semi-2dmhd-sketch},
and the same fractional Gr\"onwall argument yields
\[
\sup_{0<t\leq \bar T}\mathcal K(t)\lesssim_{C_{NS},\bar T}\|H_0\|_{X_0}.
\]
This proves the stated $X$-bound.

Existence follows by approximating $H_0$ with smooth divergence-free data,
solving \eqref{eq:Hlin-linear} classically for the approximants, and using the
uniform $X$-bound above to pass to the limit. Uniqueness follows by applying the
same estimate to the difference of two solutions with the same initial data.
\end{proof}

The next lemma isolates the elementary persistence mechanism for the
$L^\infty$ lower bound.

\begin{lemma} 
\label{lem:linfty-persistence-2dmhd}
Let $X$ be the space introduced in Section~\ref{sec:corrector}, and let
$V\in X$. Suppose $U$ is a vector field on $(0,\bar T]\times\mathbb T^2$ for
which there exist a constant $c_*>0$ and a sequence $t_n\to0+$ such that
\[
\|U(t_n)\|_{L^\infty}\ge c_*t_n^{-1/2}.
\]
Then
\[
\|U(t_n)+V(t_n)\|_{L^\infty}\ge \frac{c_*}{2}t_n^{-1/2}
\]
for all sufficiently large $n$.
\end{lemma}

\begin{proof}
Since $V\in X$, we have
\[
\|V(t_n)\|_{L^\infty}\le t_n^{-(1-\alpha)/2}\|V\|_X
= t_n^{-1/2}\big(t_n^{\alpha/2}\|V\|_X\big)
= o(t_n^{-1/2}).
\]
Hence, for all sufficiently large $n$,
\[
\|V(t_n)\|_{L^\infty}\le \frac{c_*}{2}t_n^{-1/2}.
\]
Using the triangle inequality, we obtain
\[
\|U(t_n)+V(t_n)\|_{L^\infty}\ge
\|U(t_n)\|_{L^\infty}-\|V(t_n)\|_{L^\infty}
\ge \frac{c_*}{2}t_n^{-1/2},
\]
as claimed.
\end{proof}

We then state the perturbative 2D MHD result directly in terms of the
initial magnetic stream function.

\begin{proposition} 
\label{prop:2d-mhd-perturbative-sketch}
Let $u^{NS}$ be a 2D Navier--Stokes solution on $[0,\bar T]$ satisfying the
hypotheses of Proposition~\ref{prop:semi-2dmhd-sketch}. Let $X$, $Y$,
$\mathcal X$, and $\mathcal Y$ denote the spaces introduced in
Section~\ref{sec:corrector}. Let $\psi_0\in C^\alpha(\mathbb T^2)$, set
\[
H_0\coloneqq \nabla^\perp\psi_0,
\]
and let $H^{lin}$ solve
\begin{equation}\label{eq:Hlin-sketch}
\partial_t H^{lin}-\Delta H^{lin}
+\div(u^{NS}\otimes H^{lin}-H^{lin}\otimes u^{NS})=0,
\qquad H^{lin}|_{t=0}=H_0.
\end{equation}
Then there exist constants $\delta_*,C_{0,\psi},C_*>0$, depending only on
$(C_{NS},\alpha,\kappa,\bar T)$, such that the following holds. If
\[
\|\psi_0\|_{C^\alpha}\leq \delta_*,
\]
then there exists a unique pair $(w,Z)\in \mathcal X$ with
\[
\|H^{lin}\|_X\leq C_{0,\psi}\|\psi_0\|_{C^\alpha},
\qquad
\|(w,Z)\|_{\mathcal X}\leq C_*\|\psi_0\|_{C^\alpha}^2
\]
satisfying
\[
\begin{split}
\partial_t w-\Delta w+\mathbb P\div\Big(
2u^{NS}\odot w+w\otimes w-H^{lin}\otimes H^{lin}
-2H^{lin}\odot Z-Z\otimes Z
\Big)&=0,\\
\partial_t Z-\Delta Z+\div(u^{NS}\otimes Z-Z\otimes u^{NS})\\
+\div\Big(
w\otimes H^{lin}-H^{lin}\otimes w+w\otimes Z-Z\otimes w
\Big)&=0,\\
(w,Z)|_{t=0}&=(0,0).
\end{split}
\]
Consequently,
\[
u=u^{NS}+w,\qquad B=H^{lin}+Z,
\]
solve the genuine 2D MHD system on $[0,\bar T]\times\mathbb T^2$. Moreover, the
magnetic field is controlled by
\[
\|H^{lin}\|_X+\|Z\|_X\lesssim \|\psi_0\|_{C^\alpha}.
\]
If, in addition, there exist a constant $c_*>0$ and a sequence $t_n\to0+$ such that
\[
\|u^{NS}(t_n)\|_{L^\infty}\ge c_*t_n^{-1/2},
\]
then
\[
\|u(t_n)\|_{L^\infty}\ge \frac{c_*}{2}t_n^{-1/2}
\]
for all sufficiently large $n$.
\end{proposition}

\begin{proof}
Substitute
\[
u=u^{NS}+w,\qquad B=H^{lin}+Z
\]
into the genuine 2D MHD system
\[
\partial_t u-\Delta u+\div(u\otimes u-B\otimes B)+\nabla p=0,
\]
\[
\partial_t B-\Delta B+\div(u\otimes B-B\otimes u)=0,
\qquad \div u=\div B=0,
\]
and use that $u^{NS}$ solves the Navier--Stokes equations and $H^{lin}$ solves
\eqref{eq:Hlin-sketch}. The remainder $(w,Z)$ then solves
\[
\partial_t w-\Delta w+\mathbb P\div\Big(
2u^{NS}\odot w+w\otimes w-H^{lin}\otimes H^{lin}
-2H^{lin}\odot Z-Z\otimes Z
\Big)=0,
\]
\[
\partial_t Z-\Delta Z+\div(u^{NS}\otimes Z-Z\otimes u^{NS})
+\div\Big(
w\otimes H^{lin}-H^{lin}\otimes w+w\otimes Z-Z\otimes w
\Big)=0.
\]

Define
\[
\Phi_u(W,Z)\coloneqq W\otimes W-H^{lin}\otimes H^{lin}
-2H^{lin}\odot Z-Z\otimes Z,
\]
\[
\Phi_B(W,Z)\coloneqq W\otimes H^{lin}-H^{lin}\otimes W+W\otimes Z-Z\otimes W,
\]
and set
\[
\Psi(W,Z)\coloneqq (\Phi_u(W,Z),\Phi_B(W,Z))\in\mathcal Y.
\]
We then define
\[
\mathcal F(W,Z)(t)\coloneqq -\int_0^t
\mathbb M_{u^{NS}}(t,t')\big(\Psi(W,Z)(t')\big)\,dt'.
\]

Proposition~\ref{prop:semi-2dmhd-sketch}, together with the same time
integration as in Section~\ref{sec:corrector}, yields
\[
\|\mathcal F(W,Z)\|_{\mathcal X}
\lesssim_{\bar T}
\|\Psi(W,Z)\|_{\mathcal Y}.
\]
The advantage of working with the magnetic field itself, rather than the
stream function $b$, is that the Lorentz-force term closes directly in the
same tensor norm $Y$. By Lemma~\ref{product_X_Y_lemma},
\[
\|\Phi_u(W,Z)\|_Y+\|\Phi_B(W,Z)\|_Y
\lesssim
\|(W,Z)\|_{\mathcal X}^2
+\|H^{lin}\|_X^2
+\|H^{lin}\|_X\|(W,Z)\|_{\mathcal X}.
\]
Hence, after enlarging the implicit constant if needed, there exists
$C_1=C_1(C_{NS},\alpha,\kappa,\bar T)$ such that
\[
\|\mathcal F(W,Z)\|_{\mathcal X}
\leq C_1\Big(
\|(W,Z)\|_{\mathcal X}^2
+\|H^{lin}\|_X^2
+\|H^{lin}\|_X\|(W,Z)\|_{\mathcal X}
\Big).
\]
The same bilinear estimate gives, for any
$(W_1,Z_1),(W_2,Z_2)\in \mathcal X$,
\[
\|\Phi_u(W_1,Z_1)-\Phi_u(W_2,Z_2)\|_Y
+\|\Phi_B(W_1,Z_1)-\Phi_B(W_2,Z_2)\|_Y
\]
\[
\lesssim
\Big(
\|H^{lin}\|_X+\|(W_1,Z_1)\|_{\mathcal X}+\|(W_2,Z_2)\|_{\mathcal X}
\Big)
\|(W_1-W_2,Z_1-Z_2)\|_{\mathcal X},
\]
and therefore
\[
\begin{split}
\|\mathcal F(W_1,Z_1)-\mathcal F(W_2,Z_2)\|_{\mathcal X}
&\leq C_1\Big(
\|H^{lin}\|_X+\|(W_1,Z_1)
\|_{\mathcal X}+\|(W_2,Z_2)\|_{\mathcal X}
\Big)\\
&\quad \cdot\|(W_1-W_2,Z_1-Z_2)\|_{\mathcal X}.
\end{split}
\]

Proposition~\ref{prop:Hlin-from-data-sketch} gives
\[
\|H^{lin}\|_X\leq C_0\|H_0\|_{X_0}.
\]
By the preceding remark, there exists a constant
$C_\psi=C_\psi(\alpha,\kappa,\bar T)>0$ such that
\[
\|H_0\|_{X_0}\leq C_\psi\|\psi_0\|_{C^\alpha}.
\]
Hence
\[
\|H^{lin}\|_X\leq C_0C_\psi\|\psi_0\|_{C^\alpha}.
\]
Set
\[
C_{0,\psi}\coloneqq C_0C_\psi,\qquad
\psi_*\coloneqq \|\psi_0\|_{C^\alpha},\qquad
h\coloneqq \|H^{lin}\|_X,
\]
and
\[
R\coloneqq 2C_1C_{0,\psi}^2\psi_*^2.
\]
If
\[
\psi_*\leq \delta_*\coloneqq \min\{1,(8C_1C_{0,\psi})^{-1}\},
\]
then $h\leq C_{0,\psi}\psi_*$ and, for every
$(W,Z)\in B_{\mathcal X}(0,R)$,
\[
\|\mathcal F(W,Z)\|_{\mathcal X}
\leq C_1\big(R^2+C_{0,\psi}^2\psi_*^2+C_{0,\psi}\psi_*R\big)\leq R,
\]
while for every
$(W_1,Z_1),(W_2,Z_2)\in B_{\mathcal X}(0,R)$,
\[
\begin{split}
&\quad\|\mathcal F(W_1,Z_1)-\mathcal F(W_2,Z_2)\|_{\mathcal X}
\leq C_1(C_{0,\psi}\psi_*+2R)\|(W_1-W_2,Z_1-Z_2)\|_{\mathcal X}\\
&\leq \tfrac12\|(W_1-W_2,Z_1-Z_2)\|_{\mathcal X}.
\end{split}
\]
Thus $\mathcal F$ is a contraction on $B_{\mathcal X}(0,R)$, and Banach's
fixed-point theorem yields a unique solution $(w,Z)\in\mathcal X$ with
\[
\|(w,Z)\|_{\mathcal X}\leq R=2C_1C_{0,\psi}^2\|\psi_0\|_{C^\alpha}^2.
\]
This proves the claim with $C_*\coloneqq 2C_1C_{0,\psi}^2$.

The reconstruction
\[
u=u^{NS}+w,\qquad B=H^{lin}+Z
\]
then solves the genuine 2D MHD system by construction. Moreover,
\[
\|H^{lin}\|_X+\|Z\|_X
\le (C_{0,\psi}+C_*)\|\psi_0\|_{C^\alpha}
\lesssim \|\psi_0\|_{C^\alpha}.
\]
Finally, if
\[
\|u^{NS}(t_n)\|_{L^\infty}\ge c_*t_n^{-1/2}
\]
along some sequence $t_n\to0+$, then the stated lower bound for $u=u^{NS}+w$
follows immediately from Lemma~\ref{lem:linfty-persistence-2dmhd}, since
$w\in \mathcal X\subset X$.
\end{proof}

Therefore, the statement of Theorem \ref{thm-2dmhd} follows from Proposition \ref{prop:2d-mhd-perturbative-sketch} and the 2D NSE blowup profile established in \cite{CDP}.

\bibliographystyle{amsrefs}
\bibliography{NSE}

@article{CDP,
      title={Instantaneous Type I blow-up and non-uniqueness of smooth solutions of the Navier-Stokes equations}, 
      author={Alexey Cheskidov and Mimi Dai and Stan Palasek},
      journal={arXiv preprint arXiv:2511.09556},
      year={2025},
}

@article{Dai2026,
      title={Instantaneous blowup and non-uniqueness of smooth solutions of MHD}, 
      author={Mimi Dai},
      journal={arXiv preprint arXiv:2604.08684},
      year={2026},
}

@article{FaracoLindbergSzekelyhidi2021,
  author    = {Faraco, Daniel and Lindberg, Sauli and Sz{\'e}kelyhidi, L{\'a}szl{\'o} Jr.},
  title     = {Bounded Solutions of Ideal {MHD} with Compact Support in Space-Time},
  journal   = {Arch. Ration. Mech. Anal.},
  volume    = {239},
  number    = {1},
  pages     = {51--93},
  year      = {2021},
  doi       = {10.1007/s00205-020-01570-y},
  url       = {https://doi.org/10.1007/s00205-020-01570-y},
  mrnumber  = {4198715}
}

@article{BLFNL,
  author    = {Bronzi, Anne C. and Lopes Filho, Milton C. and Nussenzveig Lopes, Helena J.},
  title     = {Wild solutions for 2D incompressible ideal flow with passive tracer},
  journal   = {Communications in Mathematical Sciences},
  volume    = {13},
  number    = {5},
  year      = {2015},
  pages     = {1333--1343},
  doi       = {10.4310/CMS.2015.v13.n5.a12},
  mrnumber  = {3344429}
}

@article{CoiculescuPalasek2025,
  title={Non-uniqueness of smooth solutions of the {N}avier--{S}tokes equations from critical data},
  author={Coiculescu, Matei P and Palasek, Stan},
  journal={Inventiones mathematicae},
  pages={1--55},
  year={2025},
  publisher={Springer}
}

@article{DeLellisSzekelyhidi2013,
  author    = {Camillo {D}e~Lellis and Sz{\'e}kelyhidi, L{\'a}szl{\'o}},
  title     = {Dissipative continuous Euler flows},
  journal   = {Inventiones Mathematicae},
  volume    = {193},
  year      = {2013},
  pages     = {377--407},
}

@article{KochTataru2001,
  author = {Koch, H. and Tataru, D.},
  title = {Well-posedness for the {N}avier–-{S}tokes equations},
  journal = {Adv. Math.},
  volume = {157},
  number = {1},
  pages = {22–35},
  year = {2001},
}

\end{document}